\documentclass{birkjour}
\usepackage[utf8]{inputenc}
\usepackage{amssymb}
\usepackage{graphicx}
\usepackage[hidelinks]{hyperref}

\renewcommand{\subjclassname}{Mathematics Subject Classification (2020)}

\newtheorem{theorem}{Theorem}
\newtheorem{proposition}[theorem]{Proposition}
\newtheorem{lemma}[theorem]{Lemma}
\newtheorem{corollary}[theorem]{Corollary}
\theoremstyle{definition}

\theoremstyle{remark}
\newtheorem{remark}[theorem]{Remark}

\newcommand{\fl}[1]{\lfloor #1 \rfloor}
\newcommand{\frc}[1]{\{ #1 \}}

\begin{document}

\title[Beatty solutions of almost Golomb equations]
      {Beatty solutions of almost Golomb\br
       functional equations}

\author[B. Cloitre]{Beno\^{\i}t Cloitre}


\thanks{ORCID:
\href{https://orcid.org/0009-0001-6778-153X}{0009-0001-6778-153X}}



\date{July 2026}

\begin{abstract}
\noindent
We study the almost Golomb equation of order~$r$,
\begin{equation*}
  a\bigl(S(n)\bigr)=n,\qquad
  S(n)=a(n)+a(n{-}1)+\cdots+a(n{-}r{+}1),
\end{equation*}
for nondecreasing sequences of positive integers. Its greedy
solution is $r$-regular in the sense of Allouche and Shallit.
Beyond that one, and for every non-square order~$r$, the
equation has another solution, an inhomogeneous
Beatty sequence $a(n)=\fl{n/\!\sqrt{r}+d}$ where~$d$ is a
parameter. No other positive slope occurs, and the equation
holds for all $n\ge r$ exactly when~$d$ lies in an explicit
closed interval depending on~$r$. That interval is a single
point when $r=2$ and has positive length for every non-square
$r\ge 3$.

Iterating the equation gives, for $k\ge 1$,
\begin{equation*}
  a^{\circ k}\bigl(S(n)\bigr)=a^{\circ(k-1)}(n),
\end{equation*}
with $a^{\circ 0}$ the identity. Its Beatty solutions of
positive slope are again of the form $\fl{n/\!\sqrt{r}+d}$,
and the sets of admissible~$d$ form an increasing chain.
The equation for~$k$ holds exactly when $a^{\circ(k-1)}$
agrees at~$n$ and at $a(S(n))$. We determine these solutions
for $k=2$ when $r=2$ and $r=3$. At the right endpoint of the
$r=2$ interval the first equation fails on a thin set of
indices, which we identify as the return times of an
irrational rotation to an explicit interval. The proofs
combine equidistribution with an exact computation in
$\mathbb{Z}[\sqrt{r}]$ over a finite range of orders.
\end{abstract}

\maketitle

\section{Introduction}\label{sec:intro}

\subsection{The almost Golomb equation}

Golomb's sequence~\cite{Golomb}
(\href{https://oeis.org/A001462}{\texttt{A001462}} in the
OEIS~\cite{OEIS}) is the unique
nondecreasing sequence of positive integers in which~$n$
appears exactly $a(n)$ times. It satisfies the equation
$a(\sum_{k=1}^{n}a(k))=n$ and grows like
$a(n)\sim\phi^{2-\phi}n^{\phi-1}$,
where $\phi=(1+\sqrt{5})/2$ is the golden ratio.

Replacing the cumulative sum by a sliding
window of size~$r$ leads to the
\emph{almost Golomb equation of order~$r$}
\begin{equation}\label{eq:AG}
  a\Bigl(\sum_{j=0}^{r-1}a(n{-}j)\Bigr)=n,
  \qquad n\ge 1,
\end{equation}
with $a(k)=0$ for $k<1$.
The earliest nondecreasing solution, called the
\emph{almost Golomb sequence} of order~$r$
in~\cite{Cloitre_AG}, is obtained inductively:
$a(n)$ is the smallest integer $\ge a(n{-}1)$
consistent with~\eqref{eq:AG}.
This finite-memory truncation changes the nature
of the sequence. The smooth power law
gives way to oscillatory linear growth, and the
greedy solution is $r$-regular in the sense of
Allouche and Shallit~\cite{AS} for every
$r\ge 2$~\cite{Cloitre_AG}.
Writing $S(n)=\sum_{j=0}^{r-1}a(n{-}j)$ for the
sliding-window sum, a notation used throughout,
equation~\eqref{eq:AG} reads $a\circ S=\mathrm{id}$.

Equation~\eqref{eq:AG} is \emph{implicit}. It
constrains~$a$ at position $S(n)$, which depends
on~$a$ itself and typically lies far beyond~$n$.
This self-reference is what leaves room for
solutions other than the greedy one.

\subsection{The greedy solution}

The greedy (earliest) solution of~\eqref{eq:AG}
is studied in detail in the companion
paper~\cite{Cloitre_AG}. It is unique, $r$-regular,
and has explicit $r$-adic denesting
formulas. For $r=2$ (sequence
\href{https://oeis.org/A394217}{\texttt{A394217}}), it begins
\[
  1,\, 2,\, 2,\, 3,\, 4,\, 4,\, 5,\, 6,\, 7,\, 7,\,
  8,\, \mathbf{8},\, 9,\, 10,\, 11,\, 12,\, 13,\,
  13,\, 14,\, 14,\, \ldots
\]
and satisfies the dyadic recurrences
$a(2n)=a(n)+a(n{+}1)-1$ and
$a(2n{+}1)=a(n)+a(n{+}1)$~\cite{Cloitre_AG}.
Its ratio $a(n)/n$ has the two limit points $2/3$
and~$3/4$, and does not converge~\cite{Cloitre_AG}.

\subsection{Motivation and related work}

Beatty sequences are a classical source of identities
involving composition. The lower and upper Wythoff
sequences $A$
(\href{https://oeis.org/A000201}{\texttt{A000201}}) and $B$
(\href{https://oeis.org/A001950}{\texttt{A001950}}),
\[
  A(n)=\fl{n\tfrac{1+\sqrt{5}}{2}},
  \qquad
  B(n)=\fl{n\tfrac{3+\sqrt{5}}{2}},
\]
are complementary and satisfy many equations built from
iterates and mixed compositions, among them
$A(A(n))=B(n)-1$. Every word in $A$ and~$B$, read as a
composition, reduces to a linear combination of $A$, $B$,
and~$1$ whose coefficients on $A$ and~$B$ are consecutive
Fibonacci numbers~\cite{KimberlingWythoff}. Hofstadter's
sequence $G(n)=n-G(G(n{-}1))$
(\href{https://oeis.org/A005206}{\texttt{A005206}}) is
itself a Beatty sequence,
equal to $\fl{(n{+}1)(\sqrt{5}-1)/2}$, and nested
recurrences of this kind are connected to morphic
words~\cite{CelayaRuskey}. Golomb's name is attached to
several other nested recurrences. The recursion
$b(b(n)+kn)=2b(n)+kn$ with $k$ a positive integer has many
increasing solutions for each~$k$~\cite{BarbeauTanny}, and
the solution space of $g(n)=g(n{-}g(n{-}1))+1$ has been
described in~\cite{SunoharaTanny}. The occurrence of a
Beatty sequence as a solution of a nested recurrence is
therefore not what is at issue here.

What equation~\eqref{eq:AG} adds is the shape of its Beatty
solution set. The equation is fixed before any Beatty form
is assumed, and a Beatty solution of positive slope is then
forced to have slope $1/\!\sqrt{r}$, at every order and for
every member of the iterated family, by
Theorem~\ref{thm:slope}. The shift is not forced. For every
non-square $r\ge 3$ an entire closed interval of shifts
solves the same equation in the full-window range $n\ge r$,
and distinct shifts give distinct sequences. In the Wythoff
setting the two slopes are determined by the
complementarity, and families of Beatty pairs are obtained
by varying the game~\cite{Larsson}.
Nested recurrences of Golomb type admitting solutions given
by Beatty functions have likewise been obtained by varying
the recurrence~\cite{IsgurKimMilcakTanny}. Here one
fixed equation carries a continuum of solutions at once,
next to the $r$-regular greedy solution.

The shift is the phase of the underlying irrational rotation.
Changing it leaves the difference word mechanical of the same
slope but moves its intercept, hence the positions of the
repeated values and the indices at which the sequence
evaluates itself. The length of the shift interval is the
margin left by the oscillation of the sawtooth divided by
$\sqrt{r}(\sqrt{r}+1)$, so it measures how far the phase may
move while the equation still holds at every index. Outside
that interval the two sides of the equation differ at some
indices, and the composed equations ask instead that an
iterate of the sequence take the same value at both.
The admissible
shift sets of the composed equations therefore increase with
the number of compositions. At the first step the increase is
proper for $r=2$ and for $r=3$, where a single point and an
interval are replaced by strictly larger intervals. Whether
it is proper at other orders, and whether the sets become
constant, are open questions taken up in
Section~\ref{sec:open}.

The same reading applies to other self-referential
recurrences. One may look first for a continuous solution,
read off the slope it forces, measure the defect that
discretization introduces, and then ask for which phases that
defect vanishes or is absorbed after finitely many
compositions. The equation studied here carries two kinds of
branch under this reading, the $r$-regular greedy solution
and the Sturmian family, and at the upper endpoint of that
family when $r=2$ the Sturmian branch carries in addition a
defect set whose gap sequence generates the minimal subshift
of a primitive substitution.

\subsection{Beatty branches and continuous shift families}

In this paper we show that, once the greedy constraint
is removed, the same implicit equation admits a continuum
of monotone solutions of a different nature in the
full-window range $n\ge r$. For every non-square $r\ge 3$
an entire interval of Beatty shifts solves it there, and
for $r=2$ such an interval appears after one further
composition. We begin with the canonical $r=2$ branch.

\begin{theorem}\label{thm:canonical}
The inhomogeneous Beatty sequence
\begin{equation}\label{eq:canonical}
  a(n)=\Bigl\lfloor\frac{n+1}{\sqrt{2}}\Bigr\rfloor,
  \qquad n\ge 1,
\end{equation}
with $a(k)=0$ for $k<1$,
satisfies the almost Golomb equation of order~$2$:
\begin{equation}\label{eq:strong}
  a\bigl(a(n)+a(n{-}1)\bigr)=n
  \qquad\text{for all } n\ge 1.
\end{equation}
\end{theorem}

This Beatty solution begins
\[
  1,\, 2,\, 2,\, 3,\, 4,\, 4,\, 5,\, 6,\, 7,\, 7,\,
  8,\, \mathbf{9},\, 9,\, 10,\, 11,\, 12,\, 12,\,
  13,\, 14,\, 14,\, \ldots
\]
Writing $u(n)=\fl{n/\!\sqrt{2}}$ for
\href{https://oeis.org/A049472}{\texttt{A049472}}, this
solution is $a(n)=u(n{+}1)$.
The two solutions agree on $[1,11]$ and diverge at
$n=12$, where the greedy repeats the value~$8$ (run of
length~$2$), while the Beatty moves to~$9$ (run of
length~$1$). This is possible because the constraint
at $n=12$ applies at position
$S(12)=a(12)+a(11)$, which is $16$ for the greedy
and~$17$ for the Beatty. Both sequences thus satisfy
the equation at $n=12$, but with the constraint
applied at different positions, allowing divergence
of the two solutions from that point on.

The two solutions also differ asymptotically, since the
Beatty ratio $a(n)/n$ converges to~$1/\!\sqrt{2}$ while the
greedy ratio does not converge.

Theorem~\ref{thm:canonical} is the case $r=2$
of the following general result.

\begin{theorem}\label{thm:universal}
Let $r\ge 2$ and define the order-$r$ Beatty
sequence
\begin{equation}\label{eq:general_beatty}
  a(n)=\Bigl\lfloor
  \frac{n}{\sqrt{r}}+\frac{\sqrt{r}}{2}
  \Bigr\rfloor,
\end{equation}
and set $S(n)=\sum_{j=0}^{r-1}a(n{-}j)$. Then, for all
$n\ge r$,
\[
  a\bigl(S(n)\bigr)=
  \begin{cases}
    n+1, & \text{if $r$ is an even perfect square},\\
    n,   & \text{otherwise}.
  \end{cases}
\]
\end{theorem}

Equivalently, the equation $a(S(n))=n$ holds for all $n\ge r$
exactly when $r$ is not an even perfect square. Odd squares
behave as non-squares. The perfect-square cases are proved in
Section~\ref{ssec:squares} (Theorem~\ref{thm:square}).

Theorem~\ref{thm:universal} concerns the range $n\ge r$.
For $1\le n<r$ the window in~\eqref{eq:AG} contains the
padding zeros, and the equation can fail. A direct computation
gives the complete classification of the full-range solutions:
$a(S(n))=n$ holds for all $n\ge 1$ exactly when
$r\in\{2,3,5,6\}$. For $7\le r\le 11$ it fails at $n=4$, where
$S(4)=7$ but $a(7)=3\ne 4$. For $r\ge 12$ it already fails at
$n=1$, since $a(1)=\fl{1/\!\sqrt{r}+\sqrt{r}/2}\ge 2$.
Consequently, for every $r\ge 7$ that is not an even perfect
square, the Beatty sequence~\eqref{eq:general_beatty} is an
eventual solution of the equation but not a full-range one,
while for $r\in\{2,3,5,6\}$ it is a solution in the full range.

The proof has three ingredients.
The continuous functional equation
identifies the unique nondecreasing affine candidate
(Section~\ref{sec:continuous}).
Floor discretization reduces the equation
to a bound on the oscillation of a walk driven by fractional
parts of multiples of $1/\!\sqrt{r}$
(Proposition~\ref{prop:oscillation}).
A block estimate for the associated window sums, refined by
reducing the walk to a shorter one of the same kind, controls
that oscillation
(Section~\ref{sec:universal}).

For every non-square order~$r$, the canonical shift
$d=\sqrt{r}/2$ is the center of the interval of admissible
shifts. Combined with the constraint on the slope
(Theorem~\ref{thm:slope}), this gives the complete
classification of the Beatty solutions of positive slope in
the full-window range $n\ge r$.

\begin{theorem}\label{thm:shift_interval}
Let $r\ge 2$ be non-square. A
Beatty sequence $a(n)=\fl{cn+d}$ with $c>0$ satisfies
\[
  a\Bigl(\textstyle\sum_{j=0}^{r-1}a(n{-}j)\Bigr)=n
  \qquad(n\ge r)
\]
if and only if
\[
  c=\frac{1}{\sqrt{r}}
  \qquad\text{and}\qquad
  d\in J_{r}
  =\Bigl[\tfrac{\sqrt{r}}{2}-\rho_{r},\,
         \tfrac{\sqrt{r}}{2}+\rho_{r}\Bigr],
  \quad
  \rho_{r}=\frac{\sqrt{r}-\Omega_{r}}{2\sqrt{r}\,(\sqrt{r}+1)},
\]
where $\Omega_{r}=\sup_{0\le\theta<1}\Phi_{r}(\theta)
-\inf_{0\le\theta<1}\Phi_{r}(\theta)$ and
$\Phi_{r}(\theta)=\sum_{j=0}^{r-1}\{\theta-j/\!\sqrt{r}\}$.
Moreover $\Omega_{r}\le\sqrt{r}$, with equality if and only if
$r=2$. Thus $J_{2}$ is a single point, while $J_{r}$ is a
nondegenerate interval for every non-square $r\ge 3$.
\end{theorem}

For odd perfect squares the slope is likewise forced, but the
rational value $c=1/\!\sqrt{r}$ replaces the dense orbit by a
periodic one and we do not classify the shifts here.

\subsection{The triple-nested family}

Applying~$a$ to both sides of~\eqref{eq:strong}
gives $a(a(S(n)))=a(n)$, the triple-nested equation
\begin{equation}\label{eq:main}
  a\bigl(a(a(n)+a(n{-}1))\bigr)=a(n).
\end{equation}
More generally, for $k\ge 1$ consider the equations
\begin{equation}\label{eq:Ek}
  (E_{k})\qquad a^{\circ k}\bigl(S(n)\bigr)=a^{\circ(k-1)}(n)
  \ \ (n\ge r),\qquad a^{\circ 0}=\mathrm{id},
\end{equation}
the first two of which, $(E_{1})$ and $(E_{2})$, are the almost
Golomb equation and the triple-nested equation~\eqref{eq:main}.
Applying~$a$ to $(E_{k})$ yields $(E_{k+1})$, so the set of
solutions grows with~$k$. We classify the Beatty solutions
of $(E_{1})$ with $c>0$ for every non-square~$r$
(Theorem~\ref{thm:shift_interval}), and those of $(E_{2})$
for $r=2$ and $r=3$ (Theorems~\ref{thm:interval}
and~\ref{thm:interval_r3}).

Write the inhomogeneous Beatty sequence of slope
$c=1/\!\sqrt{2}$ and shift~$d$ as
\[
  a(n)=\fl{cn+d}.
\]
The canonical
solution~\eqref{eq:canonical}
corresponds to $d=\sqrt{2}/2$, the unique shift for which
the strong equation~\eqref{eq:strong} holds. This case is
special, since among non-square orders $r=2$ is the only one
for which the shift is unique
(Theorem~\ref{thm:shift_interval}).
The weaker triple-nested
equation~\eqref{eq:main}, by contrast, admits an interval of
shifts. The paper determines this interval for $r=2$, and the
analogous interval for $r=3$.

\begin{theorem}\label{thm:interval}
A Beatty sequence $a(n)=\fl{cn+d}$ with $c>0$ satisfies
the triple-nested equation~\eqref{eq:main} for all $n\ge 2$
if and only if
\[
  c=\tfrac{1}{\sqrt{2}}\qquad\text{and}\qquad
  d\in I=\bigl[\sqrt{2}/2,\;2(\sqrt{2}-1)\bigr].
\]
\end{theorem}

The slope condition is Theorem~\ref{thm:slope} (with $k=2$ and
$r=2$). We determine the shift below.

The analogous theorem for $r=3$, stated and
proved in Section~\ref{sec:interval_r3}, gives, for a
Beatty sequence with $c>0$,
\[
  \begin{aligned}
    &a(a(a(n)+a(n{-}1)+a(n{-}2)))=a(n)
    \ \text{for all } n\ge 3\\
    &\qquad\iff
    c=\tfrac{1}{\sqrt{3}},\ \
    d\in I_{3}=\Bigl[\tfrac{\sqrt{3}+3}{6},\;
    \tfrac{5\sqrt{3}-3}{6}\Bigr].
  \end{aligned}
\]
Both intervals are determined exactly by the same method, a
case analysis over the finitely many configurations of the
Sturmian rotation combined with equidistribution of the
irrational orbit. Convergents of $c$ play no role in the
proof.

The right endpoint $d=2(\sqrt{2}{-}1)$ for $r=2$
admits a simple algebraic form. At this boundary
the strong equation fails on a sparse defect
set. Theorem~\ref{thm:absorption} identifies
this defect set as the return-time set of the
Sturmian rotation $\theta\mapsto\frc{\theta+c}$
to the interval $[1{-}c,1/2]$, which explains
every failure by absorption into a repeated value
of the endpoint Beatty sequence. A direct first-return
analysis, together with the return-time form of the
three-distance theorem, yields the three gap values
$\{3,4,7\}$, and the
gap sequence generates the minimal subshift of an
explicit primitive substitution with Perron
eigenvalue $1+\sqrt{2}$
(Theorem~\ref{thm:defect_subst}). The proof is an
exact self-induction of the first-return map, with
all computations in $\mathbb{Q}(\sqrt{2})$.

The equation, the interval $I$, and the defect
structure are naturally expressed in the
Pell--Ostrowski setting of
Fokkink~\cite{Fokkink}, which generalizes the
Wythoff array of Conway and Ryba to the recurrence
$X_{n+1}=2X_{n}+X_{n-1}$. Section~\ref{sec:pell_ostrowski}
makes this connection explicit. For the family $r=s^{2}+1$
with $s\ge 1$, the convergent denominators of $1/\!\sqrt{r}$
obey the recurrence of Fokkink's tower of parameter~$2s$.

\subsection{Outline}

Sections~\ref{sec:continuous}--\ref{sec:universal}
prove the existence and classification results.
Section~\ref{sec:continuous} establishes the exact
continuous solution and shows that the slope of a Beatty
solution is forced (Theorem~\ref{thm:slope}).
Section~\ref{sec:universal} reduces the discrete
equation to an oscillation bound for a sawtooth
walk, controls the walk by block estimates at a
second scale together with a finite exact computation,
and classifies the admissible shifts
(Theorem~\ref{thm:shift_interval}).
Section~\ref{sec:interval} classifies the Beatty solutions of
the triple-nested equation~\eqref{eq:main} at $r=2$
(Theorem~\ref{thm:interval}) and at $r=3$
(Theorem~\ref{thm:interval_r3}), and describes their
Pell--Ostrowski interpretation. It closes with a defect
criterion for all the iterated equations~$(E_{k})$ and shows
that their admissible shift sets form an increasing chain
(Proposition~\ref{prop:composed}, Corollary~\ref{cor:chain}).
The remaining sections treat structural
features specific to $r=2$.
Section~\ref{sec:structure} describes the Sturmian
structure and the bifurcation behavior.
Section~\ref{sec:defect} proves that the defect
set at the upper endpoint of the Beatty interval
is the return-time set of a Sturmian rotation to
an explicit sub-interval (Theorem~\ref{thm:absorption}),
computes its density and the three gap values, and
proves that the gap sequence generates the minimal
subshift of a primitive substitution on $\{3,4,7\}$
whose Perron eigenvalue is the silver ratio
$1+\sqrt{2}$ (Theorem~\ref{thm:defect_subst}).

\section{The continuous model and the slope}\label{sec:continuous}

In the discrete equation~\eqref{eq:AG}, the
sliding-window sum $S(n)$ makes the equation read
$a\circ S=\mathrm{id}$.
The continuous analogue replaces~$a$ by a continuous
nondecreasing function~$\varphi$:
\begin{equation}\label{eq:continuous}
  \varphi\Bigl(\sum_{j=0}^{r-1}\varphi(x{-}j)\Bigr)=x.
\end{equation}
Equation~\eqref{eq:continuous} is an iterative functional
equation in a single variable, in the sense of Kuczma,
Choczewski, and Ger~\cite{KCG}.
A natural candidate is the affine function
\begin{equation}\label{eq:fsol}
  \varphi(x)=\frac{x}{\sqrt{r}}+\frac{\sqrt{r}-1}{2}.
\end{equation}

\begin{theorem}\label{thm:continuous}
The function~\eqref{eq:fsol} is the unique nondecreasing
affine solution of~\eqref{eq:continuous}. The only other
affine solution is the decreasing function
$\varphi_{-}(x)=-x/\!\sqrt{r}-(\sqrt{r}+1)/2$.
Moreover, every continuous nondecreasing solution
is a homeomorphism of~$\mathbb{R}$ satisfying
the inverse equation
\begin{equation}\label{eq:inverse}
  \varphi^{-1}(x)=\sum_{j=0}^{r-1}\varphi(x{-}j).
\end{equation}
\end{theorem}

\begin{proof}
Suppose first that $\varphi(x)=cx+d$ is affine. Then
\[
  \sum_{j=0}^{r-1}\varphi(x{-}j)
  =r(cx+d)-\frac{cr(r{-}1)}{2}
\]
and then
\[
  \varphi\Bigl(\sum\varphi(x{-}j)\Bigr)
  =rc^{2}x+\Bigl(rcd-\frac{rc^{2}(r{-}1)}{2}+d\Bigr).
\]
This equals~$x$ if and only if $rc^{2}=1$
and $d(rc+1)=(r{-}1)/2$. The first equation gives
$c=\pm 1/\!\sqrt{r}$. For $c=1/\!\sqrt{r}$ one obtains
$d=(\sqrt{r}-1)/2$, the nondecreasing
solution~\eqref{eq:fsol}. For $c=-1/\!\sqrt{r}$ one obtains
$d=-(\sqrt{r}+1)/2$, the decreasing solution~$\varphi_{-}$.

Next, every continuous nondecreasing solution is a
homeomorphism satisfying the inverse equation.
Let $\varphi$ be any such solution and set
$g(x)=\sum_{j=0}^{r-1}\varphi(x{-}j)$.
Then $g$ is continuous and nondecreasing, and
$\varphi(g(x))=x$ for all~$x$,
so $\varphi$ is surjective and $g$ is injective.
Since $g$ is continuous, nondecreasing, and injective,
it is strictly increasing.
If some level set $\varphi^{-1}(\{\lambda\})=[u,v]$ were
nondegenerate ($u<v$), then $\varphi(g(x))=x$ forces
$g(x)\notin[u,v]$ for $x\ne\lambda$, so $g(x)<u$ for $x<\lambda$
and $g(x)>v$ for $x>\lambda$. Continuity of~$g$ at~$\lambda$
would then give $u\ge g(\lambda)\ge v$, contradicting $u<v$.
So $\varphi$ is strictly increasing, hence a
homeomorphism, and $g=\varphi^{-1}$.
\end{proof}

\begin{remark}
The discrete Beatty sequence
$a(n)=\fl{n/\!\sqrt{r}+\sqrt{r}/2}$ is the floor
discretization of~\eqref{eq:fsol}. Its shift $\sqrt{r}/2$
equals $(\sqrt{r}-1)/2+1/2$, and the added $1/2$ compensates
the mean of the fractional part.
\end{remark}

Whether the inverse equation~\eqref{eq:inverse}
already forces $\varphi$ to be affine remains open
(see Section~\ref{sec:open}).

\subsection{The slope of a Beatty solution}

The continuous model forces the slope $1/\!\sqrt{r}$ among
nondecreasing affine solutions (Theorem~\ref{thm:continuous}).
The following growth argument shows that the same slope is
forced for every Beatty solution of positive slope, for every
$k\ge 1$ in~\eqref{eq:Ek}.

\begin{theorem}\label{thm:slope}
Let $r\ge 2$, $k\ge 1$, $c>0$, $d\in\mathbb{R}$, and
\[
  a(n)=\fl{cn+d},\qquad
  S_{a}(n)=\sum_{j=0}^{r-1}a(n{-}j),
\]
and set $a^{\circ 0}=\mathrm{id}$. If
\[
  a^{\circ k}\bigl(S_{a}(n)\bigr)=a^{\circ(k-1)}(n)
\]
for infinitely many positive integers~$n$, then
$c=1/\!\sqrt{r}$.
\end{theorem}

\begin{proof}
From $a(m)=cm+d-\frc{cm+d}$ we have $a(m)=cm+O(1)$, with
implied constant at most $|d|+1$. Since $c>0$, every argument
appearing below tends to $+\infty$ with~$n$, so this estimate
applies to each of the finitely many iterates in turn. Summing
over the window, $S_{a}(n)=rcn+O(1)$, and applying~$a$
repeatedly,
\[
  a^{\circ k}\bigl(S_{a}(n)\bigr)=rc^{\,k+1}n+O(1),
  \qquad
  a^{\circ(k-1)}(n)=c^{\,k-1}n+O(1).
\]
If these agree for infinitely many~$n$, then
$c^{\,k-1}(rc^{2}-1)\,n=O(1)$ along an unbounded set of
integers, forcing $c^{\,k-1}(rc^{2}-1)=0$. Since $c>0$, this
gives $rc^{2}=1$, that is $c=1/\!\sqrt{r}$.
\end{proof}

The hypothesis $c>0$ cannot be dropped. For $k\ge 2$ every
positive constant sequence satisfies
$a^{\circ k}(S_{a}(n))=a^{\circ(k-1)}(n)$ trivially. Taking
$k=1$ shows that a Beatty solution of the almost Golomb
equation with $c>0$ has slope $1/\!\sqrt{r}$. The admissible
shifts are determined in Section~\ref{sec:universal}
(Theorem~\ref{thm:shift_interval}).

\section{The general equation}\label{sec:universal}

Throughout this section we fix $r\ge 2$ and write
$a$ for the order-$r$ Beatty sequence
\begin{equation*}  a(n)=\Bigl\lfloor
  \frac{n}{\alpha}+\frac{\alpha}{2}
  \Bigr\rfloor,
\end{equation*}
where $\alpha=\sqrt{r}$ and $c=1/\alpha$.
Set
\[
  S(n)=\sum_{j=0}^{r-1}a(n{-}j).
\]
We now prove Theorem~\ref{thm:universal}. The non-square case
is the equation $a(S(n))=n$ for all $n\ge r$. The
perfect-square cases are deferred to
Section~\ref{ssec:squares}.

The proof occupies Sections~\ref{ssec:discretization}
through~\ref{ssec:squares}.
Section~\ref{ssec:discretization} reduces the
equation $a(S(n))=n$ to a uniform bound on the
sawtooth function $\Phi$
(Proposition~\ref{prop:reduction}).
Section~\ref{ssec:walk} rewrites that bound as the
oscillation of a palindromic walk~$P$, recording the equation
$\Omega_{r}=1+\operatorname{osc}P$ for
$\Omega_{r}=\sup\Phi-\inf\Phi$.
Section~\ref{ssec:blocks} reduces the walk to a second,
shorter walk of the same type, and bounds the oscillation for
the boundary families $t\in\{1,2,2s{-}1,2s\}$ and, for
$s\ge 74$, for the middle range $3\le t\le 2s-2$. A finite
exact computation covers the remaining values of~$r$.
Section~\ref{ssec:completion} assembles these bounds into
$\Omega_{r}\le\sqrt{r}$.
Section~\ref{ssec:squares} treats the perfect-square
cases separately via Hermite's identity.

\subsection{Reduction to a bound on \texorpdfstring{$\Phi$}{Phi}}
\label{ssec:discretization}

By Theorem~\ref{thm:continuous}, the continuous
equation $\varphi(\sum\varphi(x{-}j))=x$ is exact.
The only obstruction to the discrete equation
$a(S(n))=n$ is the rounding error introduced
by the floor.
Define the sawtooth function
\[
  \Phi(n)
  :=\sum_{j=0}^{r-1}
  \Bigl\{\frac{n{-}j}{\alpha}+\frac{\alpha}{2}
  \Bigr\}.
\]

\begin{proposition}\label{prop:reduction}
For all $n\ge r$,
\[
  a(S(n))=n
  \quad\iff\quad
  \frac{r-\alpha}{2}
  <\Phi(n)
  \le\frac{r+\alpha}{2}.
\]
\end{proposition}

\begin{proof}
Write $x_{j}=(n{-}j)/\alpha+\alpha/2$, so that
$a(n{-}j)=x_{j}-\{x_{j}\}$.
Summing over $j=0,\ldots,r{-}1$:
\[
  S(n)=\sum x_{j}-\Phi(n).
\]
Since $\alpha^{2}=r$ we obtain
$\sum x_{j}=(n+1/2)\alpha$, hence
$S(n)=(n+1/2)\alpha-\Phi(n)$.
Applying~$a$:
\[
  a(S(n))
  =\Bigl\lfloor
  n+\frac{1+\alpha}{2}
  -\frac{\Phi(n)}{\alpha}
  \Bigr\rfloor.
\]
This equals~$n$ iff
$0\le(1{+}\alpha)/2-\Phi(n)/\alpha<1$,
which rearranges to $(r{-}\alpha)/2<\Phi(n)
\le(r{+}\alpha)/2$.
\end{proof}

\subsection{The palindromic walk and the oscillation criterion}
\label{ssec:walk}

Write $s=\fl{\alpha}$ and $\beta=\alpha-s\in(0,1)$, and set
$\Sigma=\sum_{j=1}^{r-1}\frc{jc}$. The extrema of the sawtooth
are governed by a palindromic walk.

\begin{lemma}\label{lem:walk}
Let $\Delta_{k}=\frc{kc}+\frc{(r{-}k)c}-1$ for
$1\le k\le r-1$ and $P(j)=\sum_{k=1}^{j}\Delta_{k}$, with
$P(0)=0$. Then:
\begin{enumerate}
\item[\textup{(a)}] $\Delta_{k}=\beta-1$ if $\frc{kc}<\beta$,
  and $\Delta_{k}=\beta$ otherwise. Hence
  $P(j)=j\beta-C(j)$ with
  $C(j)=\#\{1\le k\le j:\frc{kc}<\beta\}$.
\item[\textup{(b)}] $\Delta_{k}=\Delta_{r-k}$, so
  $P(j)+P(r{-}1{-}j)=P(r{-}1)$ and
  $\max_{0\le j\le r-1}P+\min_{0\le j\le r-1}P=P(r{-}1)$.
\item[\textup{(c)}] $P(r{-}1)=2\Sigma-(r{-}1)$.
\end{enumerate}
\end{lemma}

\begin{proof}
Since $kc+(r{-}k)c=\alpha=s+\beta$, the fractional part
$\frc{(r{-}k)c}$ equals $\beta-\frc{kc}$ when $\frc{kc}<\beta$
and $1+\beta-\frc{kc}$ when $\frc{kc}>\beta$. Hence
$\frc{kc}+\frc{(r{-}k)c}\in\{\beta,\,1+\beta\}$ and
$\Delta_{k}\in\{\beta-1,\,\beta\}$ accordingly, which gives
(a) on summing. The quantity $\frc{kc}+\frc{(r{-}k)c}$ is
symmetric under $k\mapsto r-k$, so $\Delta_{k}=\Delta_{r-k}$
and $P(j)+P(r{-}1{-}j)=P(r{-}1)$. Thus the multiset
$\{P(j):0\le j\le r-1\}$ is symmetric about $P(r{-}1)/2$ and
$\max P+\min P=P(r{-}1)$, which is (b). Finally, summing
$\Delta_{k}$ over $1\le k\le r-1$ and reindexing the second
term by $k\mapsto r-k$,
\[
  P(r{-}1)=\sum_{k=1}^{r-1}\frc{kc}
  +\sum_{k=1}^{r-1}\frc{(r{-}k)c}-(r{-}1)=2\Sigma-(r{-}1),
\]
which is (c).
\end{proof}

The local maximum of the sawtooth just before
the jump at $\theta=\frc{j_{0}c}$, where
$j_{0}\in\{0,1,\ldots,r{-}1\}$ indexes the jump points,
is
\begin{equation}\label{eq:local_max}
  M(j_{0})=r-j_{0}
  +\sum_{k=1}^{j_{0}}\frc{kc}
  -\sum_{m=1}^{r-1-j_{0}}\frc{mc},
\end{equation}
and the local minimum just after the same
jump is $M(j_{0})-1$
(using $\frc{-mc}=1-\frc{mc}$ for $m\ge 1$).
Figure~\ref{fig:sawtooth} shows the sawtooth over one period
for $r\in\{2,3,5\}$.

\begin{figure}[!htbp]
\centering
\includegraphics[width=\textwidth]{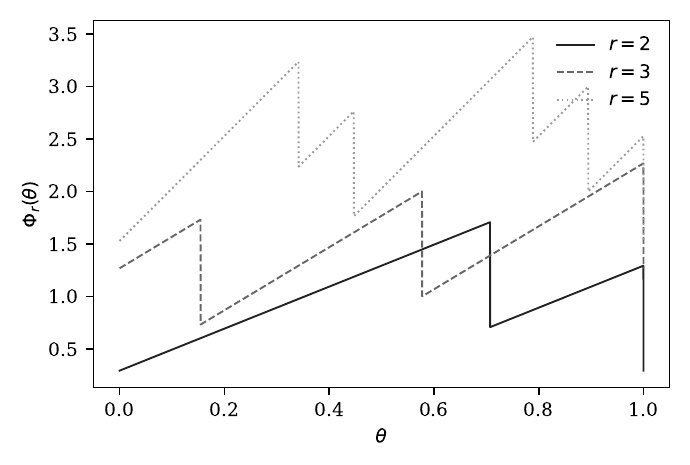}
\caption{The sawtooth $\Phi_{r}(\theta)
=\sum_{j=0}^{r-1}\{\theta-jc\}$ for $r\in\{2,3,5\}$ with
$c=1/\!\sqrt{r}$, over one period $\theta\in[0,1)$.}
\label{fig:sawtooth}
\end{figure}

The supremum and infimum of the sawtooth
$\Phi$ enjoy a palindromic symmetry.

\begin{lemma}\label{lem:palindrome}
$\sup\Phi+\inf\Phi=r$.
\end{lemma}

\begin{proof}
Expanding $P(j_{0})=\sum_{k=1}^{j_{0}}\Delta_{k}$ and
cancelling the common terms in~\eqref{eq:local_max} gives
$M(j_{0})=r-\Sigma+P(j_{0})$, so that
$\sup\Phi=r-\Sigma+\max P$ and
$\inf\Phi=r-1-\Sigma+\min P$ (the infimum being the post-jump
value $M(j_{0})-1$). Adding these and using
$\max P+\min P=P(r{-}1)=2\Sigma-(r{-}1)$
(Lemma~\ref{lem:walk}),
\[
  \sup\Phi+\inf\Phi=2r-1-2\Sigma+P(r{-}1)=r.\qedhere
\]
\end{proof}

\begin{remark}
The proof gives
$\sup\Phi=r-\Sigma+\max P$ and
$\inf\Phi=r-1-\Sigma+\min P$. Writing
$\operatorname{osc}P=\max P-\min P$, subtraction gives the
identity
\[
  \Omega_{r}:=\sup\Phi-\inf\Phi=1+\operatorname{osc}P.
\]
By Proposition~\ref{prop:reduction}, the universal equation
$a(S(n))=n$ is therefore equivalent to the oscillation bound
$\Omega_{r}\le\sqrt{r}$, which is strict for non-square
$r\ge 3$ and an equality for $r=2$, where the equation still
holds because the extrema of~$\Phi$ are not attained along
the orbit. The remainder of this section establishes this
bound.
\end{remark}

The exact criterion, in both directions, is the following.

\begin{proposition}\label{prop:oscillation}
Let $r\ge 2$ be a non-square, and let $\Sigma$, $\Delta_{k}$
and $P(j)$ be as in Lemma~\ref{lem:walk}. Then $a(S(n))=n$
for all $n\ge r$ if and only if
\begin{equation}\label{eq:osc_bound}
  \max_{0\le j\le r-1}P(j)-\min_{0\le j\le r-1}P(j)\le\alpha-1.
\end{equation}
If the left-hand side of~\eqref{eq:osc_bound} exceeds
$\alpha-1$, the equation fails at infinitely many~$n$.
\end{proposition}

\begin{proof}
By Proposition~\ref{prop:reduction}, the equation
$a(S(n))=n$ for all $n\ge r$ is equivalent to
$(r{-}\alpha)/2<\Phi(n)\le(r{+}\alpha)/2$ for all
such~$n$. The value $\Phi(n)$ depends on $n$ only through
$\theta=\{n/\alpha+\alpha/2\}$, and $\theta$ is
equidistributed in $[0,1)$ by Weyl's theorem, using that
$c=1/\alpha$ is irrational. The orbit avoids every jump point
$\{j/\alpha\}$, $0\le j\le r-1$. Indeed, if
$\{n/\alpha+\alpha/2\}=\{j/\alpha\}$, then $(2n-2j+r)/\alpha$
would be an even integer, impossible for $n\ge r$ since
$\alpha=\sqrt{r}$ is irrational and $2n-2j+r>0$. Since $\Phi$
has slope $r$ between consecutive jumps, its infimum is
attained only at jump points, whereas its supremum is
approached from the left of a jump but never attained. Hence
the two bounds hold along the orbit if and only if
\[
  \inf\Phi\ge\frac{r-\alpha}{2}
  \qquad\text{and}\qquad
  \sup\Phi\le\frac{r+\alpha}{2}\,.
\]
Indeed, a strict violation of either inequality persists on a
nonempty open interval, whereas equality at the lower endpoint
can occur only at an avoided jump point. By
Lemma~\ref{lem:palindrome}, $\sup\Phi+\inf\Phi=r$, so these two
weak inequalities are equivalent, and the equation holds if
and only if $\sup\Phi\le(r{+}\alpha)/2$.

By~\eqref{eq:local_max}, $\sup\Phi=r-\Sigma+\max P$. By
Lemma~\ref{lem:walk}(c), $P(r{-}1)=2\Sigma-(r{-}1)$, so
\[
  \Sigma-\frac{r-\alpha}{2}=\frac{\alpha-1+P(r{-}1)}{2},
\]
and Lemma~\ref{lem:walk}(b) gives
$\max P+\min P=P(r{-}1)$. Hence $\sup\Phi\le(r{+}\alpha)/2$ is
equivalent to $2\max P-P(r{-}1)\le\alpha-1$, that is
$\max P-\min P\le\alpha-1$, which is~\eqref{eq:osc_bound}. If
the left-hand side of~\eqref{eq:osc_bound} exceeds $\alpha-1$,
then $\sup\Phi>(r{+}\alpha)/2$, the upper bound on
$\Phi$ fails on a nonempty open set of $\theta$, and
equidistribution produces infinitely many failures.
\end{proof}

The oscillation $\operatorname{osc}P=\max P-\min P$ and the
resulting extrema $\sup\Phi,\inf\Phi$ for the small cases
$r\in\{2,3,5,6,7,8\}$ are collected below. The margin
$\alpha-1-\operatorname{osc}P=\alpha-\Omega_{r}$ is positive
for every non-square $r\ge 3$ and vanishes at the extremal
case $r=2$.

\begin{table}[htbp]
\centering
\small
\begin{tabular}{r|ccc|cc}
\hline
$r$ & $\operatorname{osc}P$ & $\alpha-1$ & margin
  & $\inf\Phi$ & $\sup\Phi$ \\
\hline
$2$ & $0.4142$ & $0.4142$ & $0.0000$ & $0.2929$ & $1.7071$ \\
$3$ & $0.5359$ & $0.7321$ & $0.1962$ & $0.7321$ & $2.2679$ \\
$5$ & $0.9443$ & $1.2361$ & $0.2918$ & $1.5279$ & $3.4721$ \\
$6$ & $0.7526$ & $1.4495$ & $0.6969$ & $2.1237$ & $3.8763$ \\
$7$ & $0.7085$ & $1.6458$ & $0.9373$ & $2.6458$ & $4.3542$ \\
$8$ & $1.2010$ & $1.8284$ & $0.6274$ & $2.8995$ & $5.1005$ \\
\hline
\end{tabular}
\caption{Direct verification for small non-square~$r$. The
strict bound $\operatorname{osc}P<\alpha-1$
(equivalently $\Omega_{r}<\alpha$) holds with the margin
$\alpha-1-\operatorname{osc}P$ shown, positive for the listed
non-square orders $3\le r\le 8$. At $r=2$ the margin is
exactly $0$ ($\Omega_{2}=\sqrt2$), the extremal case. There
$\inf\Phi=(r-\alpha)/2$, attained only at a jump point avoided
by the orbit, and $\sup\Phi=(r+\alpha)/2$, not attained, so the
equation still holds. For the listed non-square orders
$3\le r\le 8$ both extrema lie strictly inside
$((r-\alpha)/2,\,(r+\alpha)/2)$, which proves the non-square
part of Theorem~\ref{thm:universal} for these orders. In all
cases $\sup\Phi+\inf\Phi=r$ (Lemma~\ref{lem:palindrome}). Displayed
values are rounded evaluations of exact quantities in
$\mathbb{Q}(\sqrt{r})$. The tightest case is $r=3$, margin
$\approx 0.196$.}
\label{tab:verification}
\end{table}

\begin{remark}\label{rem:t_one}
The smallest boundary family is $t=1$, that is $r=s^{2}+1$
for some $s\ge 1$ (sequence
\href{https://oeis.org/A002522}{\texttt{A002522}}), the
values for which $\sqrt{r}$ has continued-fraction expansion
$[s;\overline{2s}]$ of minimal period one. For this family
the walk is simplest. Here $\inf\Phi$ is attained at the jump
index $j_{0}=0$, and the oscillation criterion holds
throughout. The values are $r=2,\,5,\,10,\,17,\,26,\,37,
\ldots$
\end{remark}

\subsection{Block structure and reduction}
\label{ssec:blocks}

The bound~\eqref{eq:osc_bound} concerns the oscillation of
the walk~$P$. The steps of~$P$ are governed by the positions
of the points $\frc{kc}$ relative to the interval
$[0,\beta)$, and these positions organize themselves into
blocks indexed by $m=\fl{kc}$. Throughout this subsection $r$
is a non-square, $t=r-s^{2}\in\{1,\ldots,2s\}$, and
$u=\fl{t/2}$. From $\alpha=s+\beta$ and $\alpha^{2}=r$ one has
$\beta(2s+\beta)=t$, hence $\beta=t/(s+\alpha)$. The parity
of~$t$ fixes the floor $\fl{s\beta}$.

\begin{lemma}\label{lem:floor_sbeta}
$\fl{s\beta}=\fl{(t{-}1)/2}$.
\end{lemma}

\begin{proof}
Since $\beta=t/(s+\alpha)$ and
$s+\alpha\in(2s,2s{+}1)$, we obtain
\[
  \frac{t}{2s{+}1}<\beta<\frac{t}{2s}.
\]
Multiplying by~$s$:
\[
  \frac{st}{2s{+}1}<s\beta<\frac{t}{2}.
\]
Since
\[
  \frac{st}{2s{+}1}-\frac{t{-}1}{2}
  =\frac{2s{-}t{+}1}{2(2s{+}1)}>0
\]
(as $t\le 2s$), we obtain
$(t{-}1)/2<s\beta<t/2$, whence
$\fl{s\beta}=\fl{(t{-}1)/2}$.
\end{proof}

The block decomposition of the walk is made precise by the
following.

\begin{lemma}\label{lem:blocks}
Let $1\le k\le r-1$ and $m=\fl{kc}$, so that $0\le m\le s$.
\begin{enumerate}
\item[\textup{(a)}] $\frc{kc}<\beta$ if and only if
$k<m\alpha+\alpha\beta$. Within each block
$\{k:\fl{kc}=m\}$, the indices with $\frc{kc}<\beta$ form an
initial run.
\item[\textup{(b)}] The run in block $m$ has length
$C_{m}=u+c_{m}$, where $c_{m}=1$ if
$\frc{m\beta}\ge\frc{s\beta}$ and $c_{m}=0$ otherwise.
\item[\textup{(c)}] Set $H(M):=P(\fl{M\alpha})$ for the
integers $M\ge 0$ with $\fl{M\alpha}\le r-1$, and
$\widetilde{C}_{M}:=\min\bigl(C_{M},\,r-1-\fl{M\alpha}\bigr)$.
Then
\[
  \max_{0\le j\le r-1}P(j)=\max_{M}H(M),
  \qquad
  \min_{0\le j\le r-1}P(j)
  =\min_{M}\bigl(H(M)-(1-\beta)\widetilde{C}_{M}\bigr).
\]
\end{enumerate}
\end{lemma}

\begin{proof}
(a) The condition $\frc{kc}<\beta$ reads $kc-m<\beta$, that
is, $k<(m+\beta)\alpha=m\alpha+\alpha\beta$.

(b) The run length is
$C_{m}=\fl{m\alpha+\alpha\beta}-\fl{m\alpha}
=\fl{\alpha\beta}+\chi$, where $\chi=1$ exactly
when $\frc{m\alpha}\ge 1-\frc{\alpha\beta}$. Since
$\alpha\beta=t-s\beta$, we have
$\frc{\alpha\beta}=\frc{-s\beta}=1-\frc{s\beta}$, so the
condition reads $\frc{m\alpha}\ge\frc{s\beta}$, and
$\frc{m\alpha}=\frc{m\beta}$ because $\alpha=s+\beta$.
Finally
$\fl{\alpha\beta}=t-\lceil s\beta\rceil=t-\fl{s\beta}-1
=\fl{t/2}$ by Lemma~\ref{lem:floor_sbeta}.

(c) Inside block $m$, the walk takes $C_{m}$ consecutive
steps $\beta-1$ followed by steps $\beta$, by (a). Its local
maxima therefore occur at the block boundaries
$\fl{M\alpha}$, with values $H(M)$, and its local minima at
the ends of the initial runs, with values
$H(M)-(1-\beta)C_{M}$. In the last block the run may be
truncated at $r-1$, whence $\widetilde{C}_{M}$.
\end{proof}

Summing the block contributions relates the values of~$H$ at
distant indices.

\begin{lemma}\label{lem:secondscale}
For all integers $0\le M_{1}<M_{2}$ with
$\fl{M_{2}\alpha}\le r-1$, write $W=M_{2}-M_{1}$. Then
\begin{equation}\label{eq:secondscale}
  H(M_{2})-H(M_{1})
  =W\bigl(s\beta-t+\fl{s\beta}\bigr)
  +N_{2}(M_{1},M_{2})
  +\beta\bigl(\fl{M_{2}\beta}-\fl{M_{1}\beta}\bigr),
\end{equation}
where
$N_{2}(M_{1},M_{2})
=\#\{M_{1}\le j<M_{2}:\frc{j\beta}<\frc{s\beta}\}$.
The mean value of the summand is zero, and the walk
$M\mapsto H(M)$ is driven by the rotation by~$\beta$ and the
interval $[0,\frc{s\beta})$, exactly as the walk $P$ is
driven by the rotation by~$c$ and the interval
$[0,\beta)=[0,\frc{rc})$.
\end{lemma}

\begin{proof}
Summing the steps of $P$ over the blocks
$M_{1},\ldots,M_{2}-1$ gives
\[
  H(M_{2})-H(M_{1})
  =\sum_{m=M_{1}}^{M_{2}-1}\bigl(B_{m}\beta-C_{m}\bigr),
  \qquad
  B_{m}:=\fl{(m{+}1)\alpha}-\fl{m\alpha}.
\]
The block sizes telescope, and $\alpha=s+\beta$ gives
$\sum_{m=M_{1}}^{M_{2}-1}B_{m}
=Ws+\fl{M_{2}\beta}-\fl{M_{1}\beta}$.
By Lemma~\ref{lem:blocks}(b),
$\sum_{m=M_{1}}^{M_{2}-1}C_{m}
=Wu+W-N_{2}(M_{1},M_{2})$.
Combining the two sums and using
$s\beta-u-1=s\beta-t+\fl{s\beta}$, which follows from
$\fl{s\beta}=\fl{(t-1)/2}$ and $u=\fl{t/2}$,
yields~\eqref{eq:secondscale}. The mean of the right-hand side
per step of~$M$ is
$(s\beta-t+\fl{s\beta})+\frc{s\beta}+\beta^{2}
=2s\beta+\beta^{2}-t=0$, by $\beta(2s+\beta)=t$.
\end{proof}

The second-scale walk of Lemma~\ref{lem:secondscale} admits a
closed companion, the sliding window sum of $s$ consecutive
fractional parts of the $\beta$-rotation.

\begin{lemma}\label{lem:window}
For $0\le A\le s$, set
$E(A)=\sum_{j=A-s}^{A-1}\frc{j\beta}$.
\begin{enumerate}
\item[\textup{(a)}] The quantity
$H(A)+E(A)+\beta\frc{A\beta}$ does not depend on~$A$. In
particular,
\[
  \max_{M}H-\min_{M}H
  \;\le\;
  \max_{0\le A\le s}E-\min_{0\le A\le s}E+\beta.
\]
\item[\textup{(b)}] $E(A)+E(s{+}1{-}A)=s-1$ for
$1\le A\le s$.
\end{enumerate}
\end{lemma}

\begin{proof}
(a) Fix $0\le A\le s-1$, and write $\chi_{A}=1$ if
$\frc{A\beta}<\frc{s\beta}$ and $\chi_{A}=0$ otherwise.
From $(A-s)\beta=A\beta-s\beta$ we obtain
$\frc{(A-s)\beta}=\frc{A\beta}-\frc{s\beta}+\chi_{A}$, so
\[
  E(A{+}1)-E(A)
  =\frc{A\beta}-\frc{(A{-}s)\beta}
  =\frc{s\beta}-\chi_{A}.
\]
Also
$\beta\frc{(A{+}1)\beta}-\beta\frc{A\beta}
=\beta^{2}-\beta\bigl(\fl{(A{+}1)\beta}-\fl{A\beta}\bigr)$,
and~\eqref{eq:secondscale} with $(M_{1},M_{2})=(A,A{+}1)$ gives
\[
  H(A{+}1)-H(A)
  =\bigl(s\beta-t+\fl{s\beta}\bigr)+\chi_{A}
  +\beta\bigl(\fl{(A{+}1)\beta}-\fl{A\beta}\bigr).
\]
Summing the three increments, the indicator and the floor
terms cancel, and the total is
$(s\beta-t+\fl{s\beta})+\frc{s\beta}+\beta^{2}
=2s\beta+\beta^{2}-t=0$. The oscillation bound follows from
$0\le\frc{A\beta}<1$.

(b) The change of index $j\mapsto-j$ maps the window of
$E(s{+}1{-}A)$ onto the window of $E(A)$. For $j\neq 0$ we
have $\frc{-j\beta}=1-\frc{j\beta}$, while the term $j=0$
contributes $0$ to both sums and lies in both windows
exactly when $1\le A\le s$. Hence
$E(s{+}1{-}A)=(s-1)-E(A)$.
\end{proof}

For the four boundary values of $t$, the window sum $E$ is
piecewise linear in $A$, and its oscillation admits a closed
form.

\begin{lemma}\label{lem:boundary}
Let $r=s^{2}+t$ be a non-square with $s\ge 3$ and
$t\in\{1,\,2,\,2s{-}1,\,2s\}$, and suppose $r\neq 10$. Then
\[
  \max_{0\le j\le r-1}P(j)-\min_{0\le j\le r-1}P(j)
  <\alpha-1,
\]
and $a(S(n))=n$ for all $n\ge r$.
\end{lemma}

\begin{proof}
By Lemma~\ref{lem:blocks} and Lemma~\ref{lem:window}(a), the
bound $\max P-\min P<\alpha-1$ follows from
\begin{equation}\label{eq:boundary_target}
  \operatorname{osc}E+(1-\beta)(u+1)<s-1,
  \qquad
  \operatorname{osc}E
  :=\max_{0\le A\le s}E-\min_{0\le A\le s}E.
\end{equation}
Appendix~\ref{app:analytic} evaluates $\operatorname{osc}E$
exactly in each of the four families and verifies
\eqref{eq:boundary_target}.
Proposition~\ref{prop:oscillation} then establishes the equation.
\end{proof}

The middle range of $t$ requires a bound on the oscillation
of $E$ for a general rotation, obtained in
Appendix~\ref{app:analytic} through the block structure of the
orbit (Lemma~\ref{lem:blockdev}).

\begin{lemma}\label{lem:middle}
Let $r=s^{2}+t$ be a non-square with $3\le t\le 2s-2$ and
$s\ge 74$. Then
\[
  \max_{0\le j\le r-1}P(j)-\min_{0\le j\le r-1}P(j)
  <\alpha-1,
\]
and $a(S(n))=n$ for all $n\ge r$.
\end{lemma}

\begin{proof}
As for Lemma~\ref{lem:boundary}, the bound follows
from~\eqref{eq:boundary_target}. The block estimate of
Lemma~\ref{lem:blockdev} bounds $\operatorname{osc}E$ through
the rotation of angle $\gamma=\min(\beta,1-\beta)$, and the
resulting inequality holds for $s\ge 74$. The computation,
with the explicit constants, is given in
Appendix~\ref{app:analytic}.
Proposition~\ref{prop:oscillation} establishes the equation.
\end{proof}

\subsection{Completion of the proof}
\label{ssec:completion}

We now assemble the oscillation bound. The block estimates of
Section~\ref{ssec:blocks} settle all but finitely many
orders, and the following exact computation covers the rest.

\begin{proposition}\label{prop:finite}
For every non-square $r$ with $3\le r\le 5473$,
\[
  \max_{0\le j\le r-1}P(j)-\min_{0\le j\le r-1}P(j)<\sqrt{r}-1.
\]
\end{proposition}

\begin{proof}
Each $P(j)$ is an element of $\mathbb{Z}[\sqrt{r}]$, so the
inequality is decided by a finite sequence of exact integer
comparisons. The computation, its minimal margin, and the
program are described in Appendix~\ref{app:finite}.
\end{proof}

\begin{proof}[Proof of Theorem~\ref{thm:universal},
non-square case]
For $r=2$, one has $\operatorname{osc}P=\alpha-1$, with
$\sup\Phi=(r{+}\alpha)/2$ and $\inf\Phi=(r{-}\alpha)/2$. By the
jump-avoidance argument in the proof of
Proposition~\ref{prop:oscillation}, these extremal values are
missed by the canonical orbit, so the equation holds. For $3\le r\le 8$,
Table~\ref{tab:verification} exhibits $\sup\Phi$ and $\inf\Phi$
explicitly. In general, the finite exact computation of
Proposition~\ref{prop:finite} covers every non-square $r$ with
$3\le r\le 5473$, Lemma~\ref{lem:boundary} settles the boundary
families $t\in\{1,2,2s{-}1,2s\}$ for every $s\ge 3$, and
Lemma~\ref{lem:middle} the middle range $3\le t\le 2s-2$ for
every $s\ge 74$. These cover every non-square $r>5473$. Either
$\fl{\sqrt{r}}=73$, so $r\in\{5474,5475\}$ with
$t\in\{2s{-}1,2s\}$ and Lemma~\ref{lem:boundary} applies, or
$\fl{\sqrt{r}}\ge 74$, where Lemmas~\ref{lem:boundary}
and~\ref{lem:middle} cover the boundary and middle ranges.
Together with Proposition~\ref{prop:finite} for
$3\le r\le 5473$, this accounts for all non-square orders. In
each case $\max P-\min P<\alpha-1$, and the equation follows
from Proposition~\ref{prop:oscillation}.
\end{proof}

\subsection{Perfect squares}
\label{ssec:squares}

When $r=s^{2}$, the parameter $\beta$ vanishes
and the counting argument of the previous
subsection degenerates.
The window sum can instead be evaluated exactly by grouping
terms modulo~$s$ and applying Hermite's identity.

\begin{theorem}\label{thm:square}
Let $r=s^{2}$ and $a(n)=\fl{n/s+s/2}$.
For all $n\ge s^{2}$:
\begin{enumerate}
\item if $s$ is odd, then
  $S(n)=sn-s(s{-}1)/2$
  and $a(S(n))=n$,
\item if $s$ is even, then
  $S(n)=sn-s(s{-}2)/2$
  and $a(S(n))=n+1$.
\end{enumerate}
\end{theorem}

\begin{proof}
Write $q=\fl{s/2}$.

Suppose first that $s=2q{+}1$ is odd.
Then $d=s/2=q+1/2\notin\mathbb{Z}$, and
$a(n)=q+\fl{(n{+}q)/s}$.
For $n\ge s^{2}$ the full window is active, so
\[
  S(n)
  =s^{2}q+\sum_{j=0}^{s^{2}-1}
  \Bigl\lfloor\frac{n-j+q}{s}\Bigr\rfloor.
\]
Write $j=ks+u$ with $0\le k,u\le s{-}1$.
Then $\fl{(n{-}ks{-}u{+}q)/s}
=\fl{(n{-}u{+}q)/s}-k$, giving
\[
  S(n)
  =s^{2}q
  +s\sum_{u=0}^{s-1}\Bigl\lfloor\frac{n{-}u{+}q}{s}
  \Bigr\rfloor
  -s\cdot\frac{s(s{-}1)}{2}.
\]
As $u$ ranges over $\{0,\ldots,s{-}1\}$,
the values $n{+}q{-}u$ hit each residue class
modulo~$s$ exactly once, so by Hermite's identity
($\sum_{v=0}^{n-1}\fl{x+v/n}=\fl{nx}$ for all
$x\in\mathbb{R}$)
\[
  \sum_{u=0}^{s-1}\Bigl\lfloor\frac{n{+}q{-}u}{s}
  \Bigr\rfloor=n+q-s+1.
\]
Substituting and using $q=(s{-}1)/2$ yields
$S(n)=sn-s(s{-}1)/2$.
Finally,
$a(S(n))
=q+\fl{n{-}q+q/s}=q+(n{-}q)=n$,
since $0\le q/s=(s{-}1)/(2s)<1$.

Suppose now that $s=2m$ is even.
Then $d=m\in\mathbb{Z}$ and
$a(n)=m+\fl{n/s}$.
Applying Hermite's identity in the same way gives
$S(n)=sn-s(s{-}2)/2$.
Then $a(S(n))
=m+\fl{n{-}(s{-}2)/2}=n+1$,
since $(s{-}2)/2=m{-}1\in\mathbb{Z}$.
\end{proof}

\begin{corollary}
For the canonical Beatty solution, the equation $a(S(n))=n$
holds for all odd perfect squares $r=s^{2}$
($s=3,5,7,\ldots$) and fails systematically
for all even perfect squares $r=(2m)^{2}$.
\end{corollary}

\subsection{The interval of shifts}

By Theorem~\ref{thm:slope} (with $k=1$), a Beatty solution of
the almost Golomb equation with $c>0$ has slope
$c=1/\!\sqrt{r}$. This section classifies the admissible
shifts, proving Theorem~\ref{thm:shift_interval}. The reduction
of Section~\ref{ssec:discretization} extends to an arbitrary
shift.

\begin{proposition}\label{prop:general_shift}
Let $\alpha=\sqrt{r}$, $c=1/\alpha$, and $a(n)=\fl{cn+d}$.
For $n\ge r$,
\[
  a(S(n))=n
  \iff
  \Phi(\theta_{n})\in\bigl(\alpha(D-1),\,\alpha D\bigr],
  \qquad D=d(\alpha+1)-\tfrac{r-1}{2},
\]
where $\theta_{n}=\frc{cn+d}$ and
$\Phi(\theta)=\sum_{j=0}^{r-1}\frc{\theta-jc}$.
\end{proposition}

\begin{proof}
Writing $a(n)=cn+d-\theta_{n}$ and summing over the window,
$cr=\alpha$ gives
$S(n)=\alpha n-\alpha(r{-}1)/2+rd-\Phi(\theta_{n})$.
Applying~$a$ and using $c\alpha=1$ and $crd=\alpha d$,
\[
  a(S(n))=n+\bigl\lfloor D-c\,\Phi(\theta_{n})\bigr\rfloor,
\]
which equals~$n$ iff $0\le D-c\Phi(\theta_{n})<1$, that is,
$\Phi(\theta_{n})\in(\alpha(D-1),\alpha D]$.
\end{proof}

The sawtooth $\Phi$ does not depend on~$d$, and
$(\theta_{n})_{n\ge r}$ is equidistributed in $[0,1)$ by
Weyl's theorem. Since $\inf\Phi\le\Phi(\theta_{n})\le\sup\Phi$
for every~$n$, the equation holds for all $n\ge r$ once
$\sup\Phi\le\alpha D$ and $\inf\Phi>\alpha(D-1)$. The next
lemma shows that, for non-square $r\ge 3$, the admissible
values of~$D$ in fact form a \emph{closed} interval, both
boundary values being admissible.

\begin{lemma}\label{lem:endpoint}
Let $r\ge 3$ be non-square, $\alpha=\sqrt{r}$, $c=1/\alpha$.
The sequence $a(n)=\fl{cn+d}$ satisfies $a(S(n))=n$ for all
$n\ge r$ if and only if $D=d(\alpha+1)-(r{-}1)/2$ lies in the
closed interval
$\bigl[\sup\Phi/\alpha,\ \inf\Phi/\alpha+1\bigr]$. In
particular both endpoints are admissible.
\end{lemma}

\begin{proof}
By Proposition~\ref{prop:general_shift} the equation for a
given~$n$ reads $\Phi(\theta_{n})\in(\alpha(D-1),\alpha D]$.
As $\inf\Phi\le\Phi(\theta_{n})\le\sup\Phi$ and the orbit
$(\theta_{n})_{n\ge r}$ is dense, the equation can hold for
all $n\ge r$ only if $\sup\Phi\le\alpha D$ and
$\inf\Phi\ge\alpha(D-1)$, that is
$D\in[\sup\Phi/\alpha,\ \inf\Phi/\alpha+1]$. For interior~$D$
the strict inequalities $\sup\Phi\le\alpha D$ and
$\inf\Phi>\alpha(D-1)$ conversely give it. We treat
the two endpoints.

\emph{Left endpoint} $D_{-}=\sup\Phi/\alpha$. The upper bound
$\Phi(\theta_{n})\le\alpha D_{-}=\sup\Phi$ holds for
every~$\theta$, and the lower bound
$\Phi(\theta_{n})>\alpha(D_{-}{-}1)=\sup\Phi-\alpha$ holds
because $\inf\Phi>\sup\Phi-\alpha$. Using
$\inf\Phi+\sup\Phi=r$ (Lemma~\ref{lem:palindrome}) this reads
$\sup\Phi<(r+\alpha)/2$, valid for non-square $r\ge 3$ by the
proof of Theorem~\ref{thm:universal}. So $D_{-}$ is
admissible.

\emph{Right endpoint} $D_{+}=\inf\Phi/\alpha+1$. Only the
lower bound is in question:
$\Phi(\theta_{n})>\alpha(D_{+}{-}1)=\inf\Phi$ fails exactly
when $\theta_{n}$ lands on the finite set where $\Phi$ attains
its infimum. The sawtooth is piecewise linear with slope~$r$
and drops by~$1$ at each point $\frc{jc}$, $0\le j\le r-1$, so
its minimum is the least post-jump value
\[
  \Phi\!\bigl(\frc{jc}\bigr)=\sum_{i=j-r+1}^{j}\frc{ic}
  =\tfrac{\alpha}{2}(2j-r+1)-K_{j},
  \qquad K_{j}=\sum_{i=j-r+1}^{j}\fl{ic},
\]
reached at an index~$j$ with $\Phi(\frc{jc})=\inf\Phi$. For
this~$j$ one has $d_{+}=(j+1-K_{j}c)/(\alpha+1)$. If
$\theta_{n}=\frc{jc}$ for some~$n$, then
$cn+d_{+}-jc\in\mathbb{Z}$. Multiplying by $\alpha(\alpha+1)$
and using $\alpha c=1$ and $\alpha^{2}=r$, the rational part
and the part linear in~$\alpha$ (independent over~$\mathbb{Q}$)
match separately, forcing
\[
  n=\frac{-K_{j}-r-j}{\,r-1\,}.
\]
Since $|i|\le r-1$ gives $|ic|\le(r-1)/\alpha<\alpha<s+1$ with
$s=\fl{\sqrt{r}}$, every $\fl{ic}\ge-(s+1)$, hence
$K_{j}\ge-r(s+1)$ and
\[
  n\le\frac{r(s+1)-r-j}{r-1}\le\frac{rs}{r-1}<r,
\]
the last inequality because $s\le\sqrt{r}<r-1$ for $r\ge 3$.
Thus no $n\ge r$ reaches the infimum, the lower bound is
strict along the orbit, and $D_{+}$ is admissible.
\end{proof}

We now prove Theorem~\ref{thm:shift_interval}, in which
$\Omega_{r}=\sup\Phi-\inf\Phi$ and $\alpha=\sqrt{r}$.

\begin{proof}[Proof of Theorem~\ref{thm:shift_interval}]
The slope is forced to $c=1/\!\sqrt{r}$ by
Theorem~\ref{thm:slope} (with $k=1$). We prove the shift
characterization and the bound $\Omega_{r}\le\alpha$. For
non-square $r\ge 3$, Lemma~\ref{lem:endpoint} gives the
admissible values of~$D$ as the closed interval
$\sup\Phi/\alpha\le D\le\inf\Phi/\alpha+1$, nondegenerate if
and only if $\Omega_{r}<\alpha$. Since
$d=(D+(r{-}1)/2)/(\alpha+1)$ is affine and increasing in~$D$,
the admissible shifts form a closed interval. Its midpoint
corresponds to
$D=\tfrac12(\sup\Phi+\inf\Phi)/\alpha+\tfrac12$, and by
$\sup\Phi+\inf\Phi=r$ (Lemma~\ref{lem:palindrome}) this gives
\[
  d=\frac{\tfrac12(r/\alpha+1)+\tfrac{r-1}{2}}{\alpha+1}
   =\frac{r/\alpha+r}{2(\alpha+1)}=\frac{\sqrt{r}}{2}.
\]
The half-width is
$\tfrac12\bigl(1-\Omega_{r}/\alpha\bigr)/(\alpha+1)=\rho_{r}$,
so $J_{r}=[\sqrt{r}/2-\rho_{r},\ \sqrt{r}/2+\rho_{r}]$. By the
proof of Theorem~\ref{thm:universal},
$\sup\Phi<(r+\alpha)/2$ for $r\ge 3$, hence
$\Omega_{r}=2\sup\Phi-r<\alpha$ and $\rho_{r}>0$. For $r=2$ one
has $\sup\Phi=(r+\alpha)/2$ and $\Omega_{2}=\sqrt{2}=\alpha$,
so $\rho_{2}=0$ and $J_{2}=\{\sqrt{2}/2\}$, admissible by
Theorem~\ref{thm:canonical}.
\end{proof}

\begin{corollary}\label{cor:uniqueness}
For non-square~$r$, the full-window almost Golomb equation
has a unique Beatty solution of positive slope if and only if
$r=2$. For every non-square $r\ge 3$ it has a nondegenerate
one-parameter family of Beatty solutions of positive slope.
\end{corollary}

\begin{remark}
For $r=3$ one has $\sup\Phi=4-\sqrt{3}$ and
$\inf\Phi=\sqrt{3}-1$, so
\[
  J_{3}=\Bigl[\,2-\tfrac{2}{\sqrt{3}},\ \tfrac{5\sqrt{3}-6}{3}\,\Bigr],
\]
centered at $\sqrt{3}/2$ with half-width
$\rho_{3}=(3\sqrt{3}-5)/(6+2\sqrt{3})$. Here $J_{r}$ is the set of shifts for which the equation holds,
whereas $I_{r}$ is the set for which the once-composed
equation~\eqref{eq:main} holds. This set is determined in
Section~\ref{sec:interval} for $r=2$ (the interval $I_{2}=I$
of Theorem~\ref{thm:interval}) and $r=3$ (the interval $I_{3}$
of Theorem~\ref{thm:interval_r3}). Whenever both are
defined, $J_{r}\subseteq I_{r}$, since the strong equation
implies the once-composed one. For $r=3$ the inclusion is
strict and the two intervals share the center~$\sqrt{3}/2$.
For $r=2$ the singleton $J_{2}=\{\sqrt{2}/2\}$ is the left
endpoint of $I_{2}$ (Theorem~\ref{thm:interval}). The relation
for higher orders remains open. In the
full range $n\ge 1$ the admissible shifts form a nonempty
subinterval of $J_{r}$ for $r\in\{3,5,6\}$, obtained by adding
to the defining conditions the finitely many truncated-window
equations with $1\le n<r$. For $r=3$ and $r=6$ these add no
constraint, and the subinterval is all of $J_{3}$, resp.\
$J_{6}$. For $r=5$ the equation at $n=2$ forces $a(2)=2$, that
is $d\ge 2-2/\!\sqrt{5}$, and the full-range admissible set is
the proper subinterval
\[
  \Bigl[\,2-\tfrac{2}{\sqrt{5}},\ \tfrac{25-\sqrt{5}}{20}\,\Bigr]
  \subsetneq J_{5},
\]
on which $(a(1),a(2),a(3),a(4))=(1,2,2,2)$ and the equation
holds for all $n\ge 1$. Thus the full-range equation has a
nondegenerate interval of Beatty solutions for each
$r\in\{3,5,6\}$, whereas the shift is unique for $r=2$.
\end{remark}

\subsection{Coexistence of solutions}

For each non-square~$r$ and each odd perfect square~$r$,
the almost Golomb equation~\eqref{eq:AG} of order~$r$ admits
at least two monotone solutions of different natures in the
range $n\ge r$. The greedy solution~\cite{Cloitre_AG} is
$r$-regular. The Beatty solution satisfies
$a(n)/n\to 1/\!\sqrt{r}$, and for $r\in\{2,3,5,6\}$ it
solves the equation in the full range $n\ge 1$.
Both satisfy the fixed-point relation $a\circ S=\mathrm{id}$
on the full-window range $n\ge r$, for different reasons. By
Corollary~\ref{cor:uniqueness} the Beatty branch is itself a
one-parameter family for every non-square $r\ge 3$. These are
solutions among many. Already for $r=2$, the full set of
monotone solutions has the cardinality of the
continuum~\cite{Cloitre_tree}.

\begin{remark}
The greedy branch is governed by finite-state
local structure and is therefore amenable to
$r$-regular methods~\cite{Cloitre_AG,AS}.
The Beatty branch is governed by irrational rotation
and the continuous functional
equation~\eqref{eq:continuous}.
\end{remark}

\section{The triple-nested equation}\label{sec:interval}

This section studies the once-composed equation~$(E_{2})$,
obtained by composing the almost Golomb
equation~\eqref{eq:AG} with~$a$. For $r=2$ it reads
$a(a(S(n)))=a(n)$ and reveals a continuous family of Beatty
solutions parametrized by a shift~$d$ ranging over an explicit
interval. We treat $r=2$ first, then $r=3$, and close in
Section~\ref{ssec:composed} with a defect criterion for all
the iterated equations~$(E_{k})$. The goal of the first part
is Theorem~\ref{thm:interval}. By Theorem~\ref{thm:slope} (with
$k=2$ and $r=2$) any Beatty solution of~\eqref{eq:main} with
$c>0$ has slope $c=1/\!\sqrt{2}$. The admissible shifts are
those for which the equation holds for all $n\ge 2$, namely
$d\in I=[\sqrt{2}/2,\,2(\sqrt{2}-1)]$.

The proof occupies Sections~\ref{ssec:continuous}%
--\ref{ssec:sufficiency}. Necessity is handled by
two failure lemmas that locate, in each of the three
ranges of $\theta_{n}=\frc{cn+d}$, a
sub-interval where the equation~\eqref{eq:main} fails
(Section~\ref{ssec:necessity}). Orbit density, via
Weyl equidistribution, then produces infinitely many
failures. Sufficiency is a case analysis of the same
ranges (Section~\ref{ssec:sufficiency}).
Section~\ref{sec:pell_ostrowski} then places the
interval~$I$ and its endpoints inside the
Pell--Ostrowski setting of Fokkink.
Section~\ref{sec:interval_r3} proves the
analogous result for $r=3$ by the same method.

\subsection{The continuous equation}
\label{ssec:continuous}

For the strong equation $\varphi(g(x))=x$
and the triple-nested $\varphi(\varphi(g(x)))=\varphi(x)$
(where $g(x)=\sum\varphi(x{-}j)$),
the continuous equation gives the same unique nondecreasing affine
solution~\eqref{eq:fsol} in both cases:
if $\varphi(g(x))=x$, composing with~$\varphi$
immediately yields $\varphi(\varphi(g(x)))=\varphi(x)$,
and conversely, injectivity of~$\varphi$ recovers
the strong equation from the triple-nested one.

In the discrete setting, the floor breaks this
equivalence. For $r=2$, the strong equation pins down the
unique shift $d=\sqrt{2}/2$ (the left endpoint of the Beatty
interval~$I$). This is special to $r=2$, since for
non-square $r\ge 3$ the strong equation already tolerates an
interval $J_{r}$ of shifts (Theorem~\ref{thm:shift_interval}).
The weaker triple-nested equation tolerates a wider range
of shifts. The floor in the outer composition
$a(\cdot)$ absorbs small errors, opening the
interval~$I$ to the right.

\subsection{Expansion}
Throughout this section, $c=1/\!\sqrt{2}$,
$d\in\mathbb{R}$, and $a(n)=\fl{cn+d}$.
Write $\theta_{n}=\frc{cn+d}\in[0,1)$, so that
$a(n)=cn+d-\theta_{n}$, and say that index~$n$
is a \emph{repeat} when $a(n{+}1)=a(n)$, and a
\emph{non-repeat} when $a(n{+}1)=a(n){+}1$.
Since $a(n{+}1)-a(n)=\fl{\theta_{n}+c}$, this amounts
to $\theta_{n}<1{-}c$ (repeat) or
$\theta_{n}\ge 1{-}c$ (non-repeat).
Set also $e(n)=-\theta_{n}\in(-1,0]$, so that
$a(n)=cn+d+e(n)$.
Here $S(n)=a(n)+a(n{-}1)$. Then
\begin{equation*}  S(n) = c(2n{-}1)+2d+e(n)+e(n{-}1).
\end{equation*}
Applying~$a$ to~$S(n)$ and using $2c^{2}=1$:
\begin{equation}\label{eq:aS}
  a(S(n)) = \frac{2n{-}1}{2} + d(2c{+}1)
  + c\bigl(e(n)+e(n{-}1)\bigr) + e(S(n)).
\end{equation}
A second application of~$a$ yields
$a(a(S(n)))= cn + R(n)$,
where
\begin{equation*}  R(n)=R_{0}+c^{2}(e(n)+e(n{-}1))
  +c\,e(S(n))+e(a(S(n))),
\end{equation*}
with constant part
\begin{equation*}  R_{0}
  =d(c+2)-\frac{c}{2}.
\end{equation*}
The equation~\eqref{eq:main} holds if and only if
$\fl{cn+R(n)}=\fl{cn+d}$ for all $n\ge 2$.

\subsection{Necessity}
\label{ssec:necessity}

The shift $d$ ranges over the real line, and the
interval~$I$ lies inside $[1-c,\,2-c)$, the set of shifts
with $a(1)=1$. The failure of~\eqref{eq:main} outside $I$ is
produced by a local argument on the three ranges of
$\theta_{n}\in[0,1)$ combined with density of the orbit
$(\theta_{n})_{n\ge 1}$. We describe the dynamics once and
read off each endpoint as a special case.

Expansion~\eqref{eq:aS} together with $2c^{2}=1$ gives
\begin{equation*}  a(S(n))=\fl{n-\tfrac{1}{2}+D-c\Theta_{2}(n)},
  \qquad D=d(2c+1),\quad \Theta_{2}(n)=\theta_{n}+\theta_{n-1}.
\end{equation*}
Writing $a(S(n))=n+m(n)$, the integer $m(n)$ is
determined by
\[
  m(n)=k \iff D-c\Theta_{2}(n)\in\Bigl[k+\tfrac{1}{2},\,k+\tfrac{3}{2}\Bigr).
\]
Two simple implications are immediate. If $m(n)\ge 1$,
then $a(S(n))\ge n+1$, so survival of the equation
$a(a(S(n)))=a(n)$ forces $a(n+1)=a(n)$, that is,
$n$ is a repeat. Symmetrically, if $m(n)\le -1$
then $n-1$ must be a repeat. These implications are
the two basic necessary conditions
\begin{itemize}
\item[(N$_{+}$)] $m(n)\ge 1\ \Rightarrow\ a(n+1)=a(n)$
($n$ is a repeat),
\item[(N$_{-}$)] $m(n)\le -1\ \Rightarrow\ a(n-1)=a(n)$
($n-1$ is a repeat).
\end{itemize}
Taken as a global statement over all $d$, (N$_{+}$) and
(N$_{-}$) are only necessary. If $|m(n)|\ge 2$, the equation
fails at~$n$ regardless of repeats. Indeed $c>1/2$ excludes
two consecutive repeats, so $a(n{+}2)\ge a(n)+1$ and
$a(n{-}2)\le a(n)-1$. If $m(n)\ge 2$ then
$a(a(S(n)))\ge a(n{+}2)>a(n)$, and if $m(n)\le -2$ then
$a(a(S(n)))\le a(n{-}2)<a(n)$. The bound $|m(n)|\le 1$ must
therefore be ensured a priori.
We show in Section~\ref{ssec:sufficiency} that this
bound holds uniformly on the interval~$I$, so that
on~$I$ the two conditions (N$_{+}$) and (N$_{-}$)
become jointly sufficient as well. The necessity
lemmas below exploit the two implications directly.

The three ranges of $\theta=\theta_{n}$ are
\[
  \mathrm{I}:\ \theta\in[0,1{-}c),\qquad
  \mathrm{II}:\ \theta\in[1{-}c,c),\qquad
  \mathrm{III}:\ \theta\in[c,1),
\]
which match the three branches of the decomposition
in Lemma~\ref{lem:absorption} below. In class~I,
$\theta_{n}<1-c$ means $n$ is a repeat. In classes~II
and~III, $n$ is a non-repeat, so (N$_{+}$) forbids
$m(n)\ge 1$ there. In class~II, furthermore,
$\theta_{n-1}=\theta-c+1\in[2-2c,1)\subset[1-c,1)$,
so $n-1$ is also a non-repeat, and (N$_{-}$) forbids
$m(n)\le -1$ on class~II.

Write $\Theta_{2}(n)$ explicitly in each class.
In class~III, $\theta_{n-1}=\theta-c\in[0,1-c)$,
hence $\Theta_{2}(n)=2\theta-c$ and
\begin{equation}\label{eq:cTheta_III}
  c\Theta_{2}(n)=2c\theta-\tfrac{1}{2}.
\end{equation}
In class~II, $\theta_{n-1}=\theta+1-c$, hence
$\Theta_{2}(n)=2\theta+1-c$ and
\begin{equation}\label{eq:cTheta_II}
  c\Theta_{2}(n)=2c\theta+c-\tfrac{1}{2}.
\end{equation}

The orbit $(\theta_{n})_{n\ge 1}$ is the forward
orbit of $\{c+d\}$ under the irrational rotation
$\theta\mapsto\frc{\theta+c}$. By Weyl's
equidistribution theorem, this orbit is
equidistributed in $[0,1)$, so it visits every
non-empty open sub-interval of $[0,1)$ at infinitely
many indices~$n$.

\begin{lemma}\label{lem:lower}
If $d<\sqrt{2}/2$, then~\eqref{eq:main}
fails at infinitely many~$n$.
\end{lemma}

\begin{proof}
At $d=c$ one has $D=c(2c+1)=2c^{2}+c=1+c$.
So $d<c$ is equivalent to $D<1+c$.
Work in class~II, where $\theta\in[1{-}c,c)$ and
$c\Theta_{2}=2c\theta+c-1/2$ by~\eqref{eq:cTheta_II}.
As $\theta\to c^{-}$,
$c\Theta_{2}\to 2c^{2}+c-1/2=c+1/2$, and so
\[
  D-c\Theta_{2}\ \longrightarrow\ D-c-\tfrac{1}{2}.
\]
The hypothesis $D<1+c$ gives $D-c-1/2<1/2$, so
the inequality $D-c\Theta_{2}(n)<1/2$ defines an
open sub-interval of class~II accumulating at
$\theta=c$. Explicitly, solving $D-c\Theta_{2}<1/2$
for $\theta$ yields
\begin{equation*}  \theta>\frac{D-c}{2c},
\end{equation*}
and the right-hand side is $<c$ by hypothesis.
For every~$n$ with $\theta_{n}$ in this
sub-interval, $m(n)\le -1$.

Both $n$ and $n-1$ are non-repeats throughout
class~II, so (N$_{-}$) is violated at each such~$n$.
Density of the orbit provides infinitely many such~$n$,
and~\eqref{eq:main} fails at infinitely many~$n$.
\end{proof}

The range above the interval is treated the same way.

\begin{lemma}\label{lem:upper}
If $d>2(\sqrt{2}-1)$, then~\eqref{eq:main}
fails at infinitely many~$n$.
\end{lemma}

\begin{proof}
Set $d_{\mathrm{R}}=2(\sqrt{2}-1)=4c-2$. Compute $D$ at $d_{\mathrm{R}}$:
\[
  D(d_{\mathrm{R}})=(4c-2)(2c+1)=8c^{2}+4c-4c-2=4-2=2,
\]
using $2c^{2}=1$. So $d>d_{\mathrm{R}}$ is equivalent to $D>2$.
Work in class~III, where $\theta\in[c,1)$ and
$c\Theta_{2}=2c\theta-1/2$ by~\eqref{eq:cTheta_III}.
As $\theta\to c^{+}$,
$c\Theta_{2}\to 2c^{2}-1/2=1/2$, and so
\[
  D-c\Theta_{2}\ \longrightarrow\ D-\tfrac{1}{2}.
\]
The hypothesis $D>2$ gives $D-1/2>3/2$, so the
inequality $D-c\Theta_{2}(n)\ge 3/2$ defines an
open sub-interval of class~III accumulating at
$\theta=c$. Explicitly, solving $D-c\Theta_{2}\ge 3/2$
for $\theta$ yields
\begin{equation*}  \theta\le\frac{D-1}{2c},
\end{equation*}
and the right-hand side is $>c$ by hypothesis.
For every~$n$ with $\theta_{n}$ in this
sub-interval, $m(n)\ge 1$.

Throughout class~III, $n$ is a non-repeat, so (N$_{+}$)
is violated. Density of the orbit supplies infinitely
many such~$n$, and~\eqref{eq:main} fails at infinitely
many~$n$.
\end{proof}

\begin{remark}
The two endpoints are included for slightly different reasons.
At the left endpoint $d=c$, one has $D=1+c$, and the
obstruction in Lemma~\ref{lem:lower} would require
\[
  \theta>\frac{D-c}{2c}=c,
\]
impossible in class~II, where $\theta<c$.

At the right endpoint $d=d_{\mathrm{R}}$, one has $D=2$, and the
obstruction in Lemma~\ref{lem:upper} can occur in class~III
only at the single point $\theta=c$. If $\{cn+d_{\mathrm{R}}\}=c$,
then
\[
  (n+3)c-2\in\mathbb{Z},
\]
impossible for $n\ge 2$, since $n+3\ne 0$ and
$c\notin\mathbb{Q}$. Thus neither endpoint is excluded by the
obstructions above, and only the right endpoint requires the
exclusion of a critical orbit point. Admissibility itself is
proved in Section~\ref{ssec:sufficiency}.
\end{remark}

\subsection{Sufficiency}
\label{ssec:sufficiency}

The sufficiency of $d\in I$ rests on the following
observation. The first nesting $a(S(n))$ may exceed
$n$ by one at repeat positions, and the outer $a$
cancels that excess because $a(n+1)=a(n)$ at every
repeat position. This is isolated as a lemma, from
which the theorem follows.

\begin{lemma}\label{lem:absorption}
Let $c=1/\!\sqrt{2}$, $a(n)=\fl{cn+d}$, and
$S(n)=a(n)+a(n-1)$. For every $d\in I$ and every
$n\ge 2$, $a(S(n))\in\{n,n+1\}$, with the finer
dichotomy
\begin{itemize}
\item if $\theta_{n}\ge 1-c$, that is, if $n$ is a
  non-repeat, then $a(S(n))=n$,
\item if $\theta_{n}<1-c$, that is, if $n$ is a repeat,
  then $a(S(n))\in\{n,n+1\}$.
\end{itemize}
\end{lemma}

\begin{proof}
Write $\theta=\theta_{n}$ and $D=d(2c{+}1)$.
We distinguish three ranges of~$\theta$.

Suppose first that $\theta\in[c,1)$.
Then $S(n)=2a(n)$ and
$a(S(n))=\fl{n+\xi}$ with
$\xi=D-2c\theta\in(D-2c,\;D-1]$.
Since $d\ge c$ gives $\xi>0$
and $d\le 4c{-}2$ gives $\xi\le 1$
with equality only when $d=4c{-}2$ and $\theta=c$
(requiring $c(n{+}3)\in\mathbb{Z}$, impossible),
$\xi\in(0,1)$ and $a(S(n))=n$.

Suppose next that $\theta\in[1{-}c,c)$.
Then $S(n)=2a(n)-1$ and
$\xi=D-c(2\theta{+}1)\in(0,1)$
by the same bounds on~$D$, hence $a(S(n))=n$.
The two sub-cases cover non-repeat positions.

Suppose finally that $\theta\in[0,1{-}c)$, so
$n$ is a repeat.
Then $\xi\in(0,2)$, whence
$a(S(n))\in\{n,n{+}1\}$.
\end{proof}

\begin{proof}[Proof of Theorem~\ref{thm:interval}, sufficiency]
Lemma~\ref{lem:absorption} gives $a(S(n))\in\{n,n+1\}$
for every $d\in I$ and $n\ge 2$, so the shift $m(n)$
introduced in Section~\ref{ssec:necessity} takes
values in~$\{0,1\}$. The a priori bound
$|m(n)|\le 1$ discussed there therefore holds on~$I$,
and the conditions (N$_{+}$) and (N$_{-}$) become
sufficient as well as necessary. Concretely, the
lemma gives $a(S(n))=n$ at non-repeat positions, so
$a(a(S(n)))=a(n)$ trivially. At repeat positions
$a(S(n))\in\{n,n+1\}$, and since $n$ is a repeat,
$a(n{+}1)=a(n)$, so $a(a(S(n)))=a(n)$ in both
subcases. The triple-nested equation therefore holds
for every $n\ge 2$ and every $d\in I$.
\end{proof}

\subsection{The Ostrowski--Pell setting of the \texorpdfstring{$r=2$}{r=2} triple-nested family}
\label{sec:pell_ostrowski}

The endpoint formulas, the return-time structure
of the defect set, and the underlying irrational
rotation are naturally expressed in the
Ostrowski--Pell setting attached to the
quadratic irrational $\tau=1+\sqrt{2}$, organized
by Fokkink~\cite{Fokkink} as a continuation of the
Wythoff array of Conway and Ryba.

\subsubsection*{The shared continued fraction}

The slope $c=1/\!\sqrt{2}$ has continued-fraction
expansion
\[
  c=[0;1,2,2,2,\ldots],
\]
whose convergents $p_{k}/q_{k}$ have denominators
\[
  q_{k}=1,3,7,17,41,99,239,577,\ldots
\]
(\href{https://oeis.org/A001333}{\texttt{A001333}}),
satisfying the recurrence
$q_{k+1}=2q_{k}+q_{k-1}$. The characteristic
equation $x^{2}=2x+1$ of this recurrence has
dominant root $\tau=1+\sqrt{2}$, the silver ratio, with
$\tau^{-1}=\sqrt{2}-1$.
Fokkink's Pell tower is built from the same
recurrence as the $h=2$ instance of the family
of bi-infinite sequences
\[
  X_{n+1}=h\,X_{n}+X_{n-1}, \qquad h\ge 1,
\]
where $h$ is a positive integer parameter of the
tower, unrelated to the shift~$d$ of the Beatty
family. The case $h=1$ recovers the Fibonacci
recurrence and the Wythoff array of Morrison and
Kimberling.

\subsubsection*{The shared numeration and the wall terms}

Every positive integer admits a unique
\emph{Ostrowski representation}
\[
  n=\sum_{k\ge 1}e_{k}q_{k}
\]
with digits $e_{k}\in\{0,1,2\}$, subject to the condition
that $e_{k}=2$ implies $e_{k-1}=0$ for $k\ge 2$.
Fokkink's tower carries its own numeration, built on the
denominators $D_{k}=1,2,5,12,\ldots$ of
$\sqrt{2}-1=[0;\overline{2}]$. The two weight sequences
satisfy the same recurrence and differ in their initial
terms, so the numeration used here is not the one attached
to the tower. Fokkink distinguishes two natural Beatty
sequences attached to the tower. The first
column~$A_{1,k}$ is the non-homogeneous Beatty
sequence
\[
  A_{1,k}=\Bigl\lfloor\frac{k\tau}{\tau-1}
  -\frac{1}{\tau(\tau-1)}\Bigr\rfloor
  =\Bigl\lfloor k\cdot(1{+}c)-\tfrac{1}{2}(2{-}\sqrt{2})\Bigr\rfloor
\]
of slope $\tau/(\tau-1)=1+c\approx 1.707$, by
Corollary~2.4 of Fokkink~\cite{Fokkink}. The sequence
$A_{k,0}$ of wall terms is the non-homogeneous Beatty
sequence
\begin{equation}\label{eq:wall_col}
  A_{k,0}=\Bigl\lfloor\frac{k\tau}{\tau+1}\Bigr\rfloor
  =\Bigl\lfloor\frac{k}{\sqrt{2}}\Bigr\rfloor
\end{equation}
of slope $\tau/(\tau+1)=1/\sqrt{2}=c$, by
Corollary~2.5 of Fokkink~\cite{Fokkink}. For the Pell tower
the wall terms form
\href{https://oeis.org/A049472}{\texttt{A049472}}.
Our canonical Beatty solution
\[
  a_{\mathrm{B}}(n)=\bigl\lfloor(n{+}1)/\!\sqrt{2}\bigr\rfloor
\]
is a shifted version of the wall terms~\eqref{eq:wall_col},
not of the first column. The difference word of
$a_{\mathrm{B}}$ is Sturmian of slope~$c$, and its
multiplicity structure gives a Rayleigh--Fraenkel
complementary partition of $\mathbb{Z}_{\ge 1}$
with slopes $1+\sqrt{2}$ and $(2+\sqrt{2})/2$
(see~\cite{Fraenkel}), the two sequences
\href{https://oeis.org/A003151}{\texttt{A003151}} and
\href{https://oeis.org/A003152}{\texttt{A003152}}.

\subsubsection*{Endpoint and defect interpretation}

The endpoints of the interval~$I$, the way the
equation fails beyond each of them, and the boundary
defect set are all expressed in this setting. The left endpoint
$d=\sqrt{2}/2$ is the unique shift for which the
strong equation $a(S(n))=n$ holds, producing
exactly the wall terms~\eqref{eq:wall_col}
(shifted). The right endpoint
$d=2(\sqrt{2}-1)=2/(1+\sqrt{2})$ is the maximal
shift for which the defect of the
first nesting is still absorbed by the repeat
structure of the underlying Sturmian difference
word. Both endpoints lie in
$\mathbb{Q}(\sqrt{2})\cap[0,1)$.
The equation~\eqref{eq:main} holds at both endpoints of~$I$
and fails just beyond each of them. This sharpness is
produced by orbit density near the single critical point
$\theta=c$, as in Lemmas~\ref{lem:lower}
and~\ref{lem:upper}.
The Pell-convergent identities
\[
  cq_{k}=p_{k}+(-1)^{k}\varepsilon_{k},\qquad
  2cp_{k}=q_{k}-2c(-1)^{k}\varepsilon_{k},
\]
with $\varepsilon_{k}=|cq_{k}-p_{k}|\to 0$, provide
explicit Ostrowski witnesses of the dense orbit
approaching~$c$.
The defect set at the upper endpoint, studied in
Section~\ref{sec:defect}, is the return-time set
of the irrational rotation
$\theta\mapsto\frc{\theta+c}$ to the explicit
interval $[1-c,1/2]$, and its gap alphabet
$\{3,4,7\}=\{q_{2},\,q_{3}-q_{2},\,q_{3}\}$ is built from two
consecutive convergent denominators and their difference.

\subsubsection*{Generalization to \texorpdfstring{$r=s^{2}+1$}{r=s\textasciicircum 2+1}}

Remark~\ref{rem:t_one} singled out the family
$r=s^{2}+1$, $s\ge 1$, the boundary case $t=1$. This is
exactly the family for which
$c=1/\!\sqrt{r}$ has continued-fraction expansion
$[0;s,\overline{2s}]$ of minimal period one, so its
convergent denominators obey the recurrence
$X_{n+1}=2s\,X_{n}+X_{n-1}$ of Fokkink's tower of parameter
$h=2s$. The initial quotient~$s$ differs from the tower
parameter, so the two numerations built on these
denominators are not the same. The Pell tower
($s=1$) is the simplest case. For every $s\ge 2$ the order
$r=s^{2}+1$ is a non-square, so its strong-equation
interval~$J_{r}$ is non-degenerate
(Theorem~\ref{thm:shift_interval}) and satisfies
$J_{r}\subseteq W_{r}^{(2)}$ (Corollary~\ref{cor:chain}). The
triple-nested equation therefore admits a non-trivial interval
of shifts for every $s\ge 1$, equal to~$I$ when $s=1$
(Theorem~\ref{thm:interval}) and containing~$J_{r}$ when
$s\ge 2$. The exact endpoints of $W_{r}^{(2)}$ and the
associated defect structures for $s\ge 2$ are left for future
work.

\subsection{The triple-nested equation for \texorpdfstring{$r=3$}{r=3}}
\label{sec:interval_r3}

The same argument extends to $r=3$. The three
sliding terms produce a larger case decomposition,
and the repeat structures of $n$ and of $n-1$ no
longer coincide, but the underlying logic, local
cases plus orbit density, gives a full
\emph{if and only if}.

Fix $r=3$, $c=1/\!\sqrt{3}$, and
$a(n)=\fl{cn+d}$. Write
$S(n)=a(n)+a(n{-}1)+a(n{-}2)$ and
$\Theta(n)=\theta_{n}+\theta_{n-1}+\theta_{n-2}$
with $\theta_{n}=\frc{cn+d}$. Since $3c^{2}=1$,
\begin{equation*}  c\,S(n)+d=n-1+D-c\Theta(n),\qquad
  D=d(\sqrt{3}+1),
\end{equation*}
and $a(S(n))=n-1+\fl{D-c\Theta(n)}$. Setting
$a(S(n))=n+m(n)$, the shift is
\begin{equation}\label{eq:m_r3}
  m(n)=\fl{D-c\Theta(n)}-1.
\end{equation}
As for $r=2$, the triple-nested equation
$a(a(S(n)))=a(n)$ implies the two basic necessary
conditions
\begin{itemize}
\item[(N$_{+}$)] $m(n)\ge 1\ \Rightarrow\ a(n+1)=a(n)$
($n$ is a repeat),
\item[(N$_{-}$)] $m(n)\le -1\ \Rightarrow\ a(n-1)=a(n)$
($n-1$ is a repeat),
\end{itemize}
but they are jointly equivalent to the equation only
under the a priori bound $|m(n)|\le 1$. We show
below that this bound holds uniformly on $I_{3}$,
so that on $I_{3}$ the necessity and the
sufficiency parts can be read off from
(N$_{+}$) and (N$_{-}$) together.

\begin{theorem}\label{thm:interval_r3}
A Beatty sequence $a(n)=\fl{cn+d}$ with $c>0$ satisfies
$a(a(S(n)))=a(n)$ for all $n\ge 3$ if and only if
$c=1/\!\sqrt{3}$ and
\begin{equation*}  d\in I_{3}=\Bigl[\,\frac{\sqrt{3}+3}{6},\;
  \frac{5\sqrt{3}-3}{6}\,\Bigr].
\end{equation*}
The interval is centered at $\sqrt{3}/2$ with half-width
$(2\sqrt{3}-3)/6$, and both endpoints lie in
$\mathbb{Q}(\sqrt{3})$.
\end{theorem}

By Theorem~\ref{thm:slope} (with $k=2$, $r=3$) the slope is
necessarily $1/\!\sqrt{3}$. This subsection characterizes the
shift. We describe the four
combinatorial cases, prove sufficiency, then prove necessity
by accumulating the obstruction at $\theta=1-c$.

\subsubsection*{The four cases}

Parametrize by $t=\theta_{n-2}\in[0,1)$, so that
$\theta_{n-1}=\frc{t+c}$ and
$\theta_{n}=\frc{t+2c}$. Since $c\in(1/2,1)$ and
$2c\in(1,2)$, the wraps occur at
$t=1-c$ for $\theta_{n-1}$ and at $t=2-2c$ and
$t=1-2c\,(<0)$ for $\theta_{n}$. A direct case
analysis against the threshold
$\theta_{\cdot}<1-c$ that defines a repeat yields
the four cases
\[
\begin{array}{l|l|l|l|l}
  \text{case} & \text{range of }t & \theta_{n-1} & \theta_{n} & \text{status of }(n,n-1) \\
  \hline
  \text{A} & [0,\,2{-}3c) & t+c & t+2c-1 & \text{$n$ repeat, $n-1$ non-repeat}\\
  \text{B} & [2{-}3c,\,1{-}c) & t+c & t+2c-1 & \text{both non-repeat}\\
  \text{C} & [1{-}c,\,2{-}2c) & t+c-1 & t+2c-1 & \text{$n$ non-repeat, $n-1$ repeat}\\
  \text{D} & [2{-}2c,\,1) & t+c-1 & t+2c-2 & \text{$n$ repeat, $n-1$ non-repeat}
\end{array}
\]

The sum $\Theta(n)=t+\theta_{n-1}+\theta_{n}$ is
piecewise affine in~$t$ and takes the values
\begin{equation}\label{eq:Theta_cases}
  \Theta(t)=\begin{cases}
    3t+3c-1,&t\in\mathrm{A}\cup\mathrm{B},\\[2pt]
    3t+3c-2,&t\in\mathrm{C},\\[2pt]
    3t+3c-3,&t\in\mathrm{D}.
  \end{cases}
\end{equation}
The ranges of $\Theta$ on the four cases are
\begin{equation}\label{eq:Theta_ranges}
\begin{aligned}
  \Theta(\mathrm{A}) &=[\sqrt{3}-1,\;5-2\sqrt{3}),\\
  \Theta(\mathrm{B}) &=[5-2\sqrt{3},\;2),\\
  \Theta(\mathrm{C}) &=[1,\;4-\sqrt{3}),\\
  \Theta(\mathrm{D}) &=[3-\sqrt{3},\;\sqrt{3}).
\end{aligned}
\end{equation}
The argument uses two extrema of~$\Theta$.

\emph{The lower extremum.} The non-repeat
positions for~$n$ are exactly $\mathrm{B}\cup\mathrm{C}$, and
\begin{equation}\label{eq:lower_extremum}
  \inf_{t\in\mathrm{B}\cup\mathrm{C}}\Theta(t)=1,
\end{equation}
attained only at the left endpoint $t=1-c$ of
case~C.

\emph{The upper extremum.} The non-repeat
positions for~$n-1$ are exactly
$\mathrm{A}\cup\mathrm{B}\cup\mathrm{D}$, and
\begin{equation}\label{eq:upper_extremum}
  \sup_{t\in\mathrm{A}\cup\mathrm{B}\cup\mathrm{D}}\Theta(t)=2,
\end{equation}
approached only at the left limit $t\to(1-c)^{-}$
from case~B.

Both extrema follow from~\eqref{eq:Theta_cases}
and~\eqref{eq:Theta_ranges} by direct inspection.

\subsubsection*{The threshold values of \texorpdfstring{$D$}{D}}

Compute $D$ at the two endpoints of $I_{3}$.
At $d_{\mathrm{L}}=(\sqrt{3}+3)/6$:
\[
  D_{\mathrm{L}}=\frac{\sqrt{3}+3}{6}\cdot(\sqrt{3}+1)
      =\frac{6+4\sqrt{3}}{6}
      =1+\frac{2}{\sqrt{3}}
      =1+2c.
\]
At $d_{\mathrm{R}}=(5\sqrt{3}-3)/6$:
\[
  \begin{aligned}
    D_{\mathrm{R}}&=\frac{5\sqrt{3}-3}{6}\cdot(\sqrt{3}+1)
      =\frac{15+5\sqrt{3}-3\sqrt{3}-3}{6}\\
    &=\frac{12+2\sqrt{3}}{6}
      =2+\frac{1}{\sqrt{3}}
      =2+c.
  \end{aligned}
\]
So $d\in I_{3}$ is equivalent to $D\in[1+2c,\,2+c]$.

\subsubsection*{Sufficiency}

Assume $d\in I_{3}$, equivalently $D\in[1+2c,\,2+c]$.
The shift $m(n)$ is then bounded a priori, so (N$_{+}$) and
(N$_{-}$) together are sufficient. The sum~$\Theta(n)$ ranges over
$\Theta(\mathrm{A})\cup\Theta(\mathrm{B})\cup
\Theta(\mathrm{C})\cup\Theta(\mathrm{D})
\subset[\sqrt{3}-1,\,4-\sqrt{3})$
by~\eqref{eq:Theta_ranges}. Multiplying by $c$ and
using $c\sqrt{3}=1$,
\[
  c\Theta(n)\in[1-c,\,4c-1),
\]
whence
\[
  D-c\Theta(n)\in(D-4c+1,\,D-1+c]
  \subset(1+2c-4c+1,\,2+c-1+c]
  =(2-2c,\,1+2c].
\]
In particular $D-c\Theta(n)\in(0,3)$, so
$\fl{D-c\Theta(n)}\in\{0,1,2\}$ and by~\eqref{eq:m_r3}
\begin{equation}\label{eq:m_bounded_r3}
  m(n)\in\{-1,0,1\}\qquad\text{for every }n\ge 3,
  \ d\in I_{3}.
\end{equation}
With~\eqref{eq:m_bounded_r3} at hand, (N$_{+}$) and
(N$_{-}$) jointly imply the triple-nested equation
on~$I_{3}$. The two conditions are verified case by case.

Requirement (N$_{+}$) forbids $m(n)\ge 1$ at
non-repeats of~$n$, which by~\eqref{eq:m_r3} is the
inequality $D-c\Theta(n)\ge 2$, equivalently
\begin{equation}\label{eq:thr_plus}
  \Theta(n)\le \frac{D-2}{c}=\sqrt{3}(D-2).
\end{equation}
The hypothesis $D\le 2+c$ gives
$\sqrt{3}(D-2)\le \sqrt{3}c=1$.

For $D<2+c$ (the interior of~$I_{3}$ on the right),
the bound in~\eqref{eq:thr_plus} is strict,
$\sqrt{3}(D-2)<1$, and~\eqref{eq:lower_extremum} gives
$\Theta(n)\ge 1>\sqrt{3}(D-2)$ on
$\mathrm{B}\cup\mathrm{C}$. Hence (N$_{+}$) holds
vacuously.

For $D=2+c$ (the right endpoint $d=d_{\mathrm{R}}$),
equality in~\eqref{eq:lower_extremum} is possible only at the single
point $t=1-c$. Writing $d_{\mathrm{R}}=(5c-1)/2$, the
equation $\{c(n-2)+d_{\mathrm{R}}\}=1-c$ unfolds to
$c(n-1)+(5c-1)/2-1\in\mathbb{Z}$. Multiplying by~$2$,
this reads $(2n+3)c=2k+3$ for some integer~$k$.
Since $2n+3\ne 0$ and $c$ is irrational, no
integer~$n$ solves this equation, the threshold is
never attained, and (N$_{+}$) holds.

Requirement (N$_{-}$) forbids $m(n)\le -1$ at
non-repeats of~$n-1$, which by~\eqref{eq:m_r3} is
the inequality $D-c\Theta(n)<1$, equivalently
\begin{equation*}  \Theta(n)>\frac{D-1}{c}=\sqrt{3}(D-1).
\end{equation*}
The hypothesis $D\ge 1+2c$ gives
$\sqrt{3}(D-1)\ge 2\sqrt{3}c=2$, and by~\eqref{eq:upper_extremum}
$\Theta(n)\le 2$ on
$\mathrm{A}\cup\mathrm{B}\cup\mathrm{D}$ with
the supremum not attained. For $D>1+2c$ the bound
is strict and (N$_{-}$) holds vacuously. At
$D=1+2c$ (the left endpoint $d=d_{\mathrm{L}}=(1+c)/2$),
the bound becomes $\Theta(n)>2$, which is impossible
on $\mathrm{A}\cup\mathrm{B}\cup\mathrm{D}$ since
the supremum~$2$ is not attained. Hence (N$_{-}$)
holds.

Sufficiency is established for every $d\in I_{3}$.

\subsubsection*{Necessity}

\emph{Right bound.} Assume $d>d_{\mathrm{R}}$, equivalently
$D>2+c$, so that $\sqrt{3}(D-2)>1$. On case~C the function
$\Theta(t)=3t+3c-2$ is continuous and increasing with
$\Theta(1-c)=1$. Hence there is an $\varepsilon>0$ with
\[
  U_{\mathrm{R}}=(1-c,\,1-c+\varepsilon)\subset\mathrm{C},
  \qquad
  \Theta(t)<\sqrt{3}(D-2)\ \ (t\in U_{\mathrm{R}}).
\]
For every orbit point in $U_{\mathrm{R}}$ one has $D-c\Theta(n)>2$,
hence $m(n)\ge 1$, while $n$ is a non-repeat (case~C). This
violates (N$_{+}$). The orbit $(\theta_{n-2})_{n\ge 3}$ is
dense in $[0,1)$ by Weyl's theorem, so it visits $U_{\mathrm{R}}$
infinitely often, and~\eqref{eq:main} fails at infinitely
many~$n$.

\emph{Left bound.} Assume $d<d_{\mathrm{L}}$, equivalently
$D<1+2c$. Then $\sqrt{3}(D-1)<2$, so the threshold
$\sqrt{3}(D-1)$ lies strictly below the supremum~$2$
on $\Theta(\mathrm{B})=[5-2\sqrt{3},2)$. The set
\[
  U_{\mathrm{L}}=\{t\in\mathrm{B}:\Theta(t)>\sqrt{3}(D-1)\}
       =\{t:t_{\star}<t<1-c\},
\]
where
$t_{\star}=\max\bigl((\sqrt{3}(D-1)+1)/3-c,\,2-3c\bigr)$,
is a non-empty open sub-interval of case~B
accumulating at the left limit of $t=1-c$. On
$U_{\mathrm{L}}$, we have $D-c\Theta(n)<1$, hence $m(n)\le -1$,
while $n-1$ is a non-repeat (case~B). This violates
(N$_{-}$). Density of the orbit yields infinitely
many failures.

\emph{Endpoints.} At $d=d_{\mathrm{R}}$, $D=2+c$ and the
threshold for $m(n)\ge 1$ is $\Theta\le 1$. The only
value in $\Theta(\mathrm{B}\cup\mathrm{C})$ satisfying this
is $\Theta=1$ at $t=1-c$, never attained. At
$d=d_{\mathrm{L}}$, $D=1+2c$ and the threshold for
$m(n)\le -1$ is $\Theta>2$, never attained on
$\mathrm{A}\cup\mathrm{B}\cup\mathrm{D}$. Both endpoints
are therefore included in~$I_{3}$.

This completes the proof of
Theorem~\ref{thm:interval_r3}.

\begin{remark}
The interval $I_{3}$ is symmetric around the
canonical shift $\sqrt{3}/2$, with half-width
$(2\sqrt{3}-3)/6$. Symmetry for $r=3$ contrasts
with the asymmetry of $I$ for $r=2$. In both cases
the endpoints arise in the same way. The threshold for
$m(n)=+1$ and the threshold for $m(n)=-1$ are
governed by extrema of~$\Theta$ over the non-repeat
sets of~$n$ and of $n-1$ respectively, and for $r=3$
these two extrema sit symmetrically about the center.
For $r=2$, the two-term window gives only one
non-repeat range coupling both constraints,
producing asymmetry.
\end{remark}

\begin{remark}
For a general non-square order~$r$, the same expansion reduces
the triple-nested equation~\eqref{eq:main} to finitely many
piecewise-affine conditions along the irrational rotation
orbit. For $r\ge 4$, however, the absorption may involve runs
of several consecutive equal values, rather than only the
repeat status of $(n,n-1)$ as for $r\in\{2,3\}$, and it is not
known whether the resulting admissible set $W_{r}^{(2)}$ is
always an interval, nor where its endpoints lie.
\end{remark}

At the right endpoint $d=(5\sqrt{3}{-}3)/6$,
the defect set begins
\[
  16,\,23,\,42,\,49,\,61,\,68,\,87,\,94,\,
  113,\,120,\ldots
\]
(\href{https://oeis.org/A395252}{\texttt{A395252}}),
with gap values numerically observed to be
$\{7,\,12,\,19\}$. The values $7$ and $19$ are consecutive
convergent denominators of the continued fraction
$[0;1,1,2,1,2,\ldots]$ of $1/\!\sqrt{3}$, and $12=19-7$. The defect set
appears numerically to have density $(2{-}\sqrt{3})/3$ in
$\mathbb{Z}_{\ge 1}$, matching the length of the rotation
interval that controls the absorption. A substitution
analogue of
Theorem~\ref{thm:defect_subst} for $r=3$ is formulated as an
open problem in Section~\ref{sec:open}.

\subsection{The iterated equations}
\label{ssec:composed}

The equation and its triple-nested form are the first two
members of the family of iterated equations $(E_{k})$ of
Section~\ref{sec:intro},
\[
  (E_{k})\qquad a^{\circ k}\bigl(S(n)\bigr)=a^{\circ(k-1)}(n)
  \quad (n\ge r),\qquad a^{\circ 0}=\mathrm{id}.
\]
By Theorem~\ref{thm:slope} a Beatty solution of $(E_{k})$
with $c>0$ has slope $c=1/\!\sqrt{r}$. For this slope, let
\[
  W_{r}^{(k)}=\bigl\{\,d\in\mathbb{R}:\
  a(n)=\fl{n/\!\sqrt{r}+d}\ \text{satisfies}\ (E_{k})
  \ \text{for all}\ n\ge r\,\bigr\},
\]
where, as throughout the paper, only shifts giving a positive
sequence are kept, namely $a(n)\ge 1$ for every $n\ge 1$.
Since the sequence is nondecreasing this is the single
requirement $a(1)\ge 1$, equivalently $d\ge 1-1/\!\sqrt{r}$.
The three
classification results of this paper read
$W_{r}^{(1)}=J_{r}$ for every non-square~$r$
(Theorem~\ref{thm:shift_interval}), $W_{2}^{(2)}=I$
(Theorem~\ref{thm:interval}), and $W_{3}^{(2)}=I_{3}$
(Theorem~\ref{thm:interval_r3}).

A single integer statistic passes from one member of the
family to the next. Let
\[
  \mu(n)=a\bigl(S(n)\bigr)-n
\]
be the signed defect of the equation at index~$n$, so that
$(E_{1})$ is the statement $\mu\equiv 0$ for all $n\ge r$. The
other members of the family reduce to the same defect, read
through a lower composition of~$a$.

\begin{proposition}\label{prop:composed}
For every $k\ge 1$,
\[
  (E_{k})\ \text{holds for all}\ n\ge r
  \iff
  a^{\circ(k-1)}\bigl(n+\mu(n)\bigr)=a^{\circ(k-1)}(n)
  \ \text{for all}\ n\ge r.
\]
\end{proposition}

\begin{proof}
By the definition of~$\mu$ one has $a(S(n))=n+\mu(n)$, hence
\[
  a^{\circ k}\bigl(S(n)\bigr)
  =a^{\circ(k-1)}\bigl(a(S(n))\bigr)
  =a^{\circ(k-1)}\bigl(n+\mu(n)\bigr).
\]
The equation $(E_{k})$ equates the left-hand side with
$a^{\circ(k-1)}(n)$, which is the stated equivalence.
\end{proof}

Applying~$a$ once more carries each member of the family to
the next, which orders the shift sets.

\begin{corollary}\label{cor:chain}
$(E_{k})$ implies $(E_{k+1})$, so the admissible shifts form
an increasing chain
\[
  W_{r}^{(1)}\subseteq W_{r}^{(2)}\subseteq
  W_{r}^{(3)}\subseteq\cdots.
\]
\end{corollary}

\begin{proof}
If $(E_{k})$ holds, then
\[
  a^{\circ(k+1)}\bigl(S(n)\bigr)
  =a\bigl(a^{\circ k}(S(n))\bigr)
  =a\bigl(a^{\circ(k-1)}(n)\bigr)
  =a^{\circ k}(n),
\]
which is $(E_{k+1})$. Every shift admissible for $(E_{k})$ is
therefore admissible for $(E_{k+1})$.
\end{proof}

Since $a$ is nondecreasing, every iterate $a^{\circ j}$ is
nondecreasing. The equality in
Proposition~\ref{prop:composed} therefore says that
$a^{\circ(k-1)}$ takes the same value at $n$ and at
$n+\mu(n)$, and hence the same value at every index lying
between them. In this reading
$(E_{k})$ asks that the displacement $n\mapsto n+\mu(n)$ leave
$a^{\circ(k-1)}$ unchanged.

For non-square~$r$ and $d\in J_{r}$ the defect vanishes, with
$\mu(n)=0$ for every $n\ge r$. For the two intervals
determined above at $k=2$, the case analyses show that
\[
  \mu(n)\in\{-1,0,1\}
\]
for every $n\ge r$, with $r=2$, $d\in I_{2}=I$, or $r=3$,
$d\in I_{3}$. Thus, for $r\in\{2,3\}$ and $d\in I_{r}\setminus J_{r}$, every
nonzero defect has size one. This bound is
established by the case analyses of
Sections~\ref{ssec:sufficiency} and~\ref{sec:interval_r3}. By
Proposition~\ref{prop:composed} the equation $(E_{2})$ says
that $a(n+\mu(n))=a(n)$ for every $n\ge r$, and where the
defect is nonzero this equates the values of~$a$ at two
consecutive indices. No analogous bound is asserted here
for arbitrary shifts or for $r\ge 4$.

If $a^{\circ(k-1)}$ takes the same value at two indices then
$a^{\circ k}$ takes the same value at those two indices, so the condition of
Proposition~\ref{prop:composed} can only become easier as~$k$
increases, in agreement with Corollary~\ref{cor:chain}.
Section~\ref{sec:open} takes up whether the increasing
chain $(W_{r}^{(k)})_{k\ge 1}$ is eventually constant and
what shape its members have.

\section{Internal structure}\label{sec:structure}

Throughout this section we consider the case $r=2$
and write $a_{\mathrm{B}}$ for the canonical Beatty
solution~\eqref{eq:canonical} and
$a_{\mathrm{G}}$ for the greedy solution.

\subsection{Sturmian words}

The difference sequence
$w_{\mathrm{B}}(n)=a_{\mathrm{B}}(n{+}1)-a_{\mathrm{B}}(n)
\in\{0,1\}$, $n\ge 1$, is the lower mechanical word of slope
$c=1/\!\sqrt{2}$ and intercept $\sqrt{2}-1$,
\[
  w_{\mathrm{B}}(n)
  =\fl{nc+(\sqrt{2}-1)}-\fl{(n{-}1)c+(\sqrt{2}-1)},
  \qquad n\ge 1.
\]
The word is Sturmian and balanced, and every two factors of
the same length differ by at most one in the number
of~$1$'s.

The multiplicity of a value~$v$, that is the number of
indices~$k$ with $a_{\mathrm{B}}(k)=v$, equals
\[
  m(v)=\Bigl\lceil\tfrac{v+1-d}{c}\Bigr\rceil
       -\Bigl\lceil\tfrac{v-d}{c}\Bigr\rceil\in\{1,2\},
  \qquad d=\tfrac1{\sqrt2},
\]
and the sequence $(m(v))_{v\ge 1}$ is Sturmian. (The
difference $w_{\mathrm{B}}(k)=a_{\mathrm{B}}(k{+}1)-a_{\mathrm{B}}(k)$
records instead whether the positions $k,k{+}1$ carry the
same value, and does not by itself give the multiplicity of
a value.)

The characteristic word of slope~$c$ is the fixed point of
the primitive morphism
\[
  \sigma_{\mathrm{P}}\colon\ 1\mapsto 110,\qquad 0\mapsto 1,
\]
whose incidence matrix has characteristic polynomial
$\lambda^{2}-2\lambda-1$ and Perron eigenvalue $1+\sqrt{2}$,
and $w_{\mathrm{B}}$ is the shift of this fixed point by one
letter, since $w_{\mathrm{B}}(n)=t(n{+}1)$ for
$t(n)=\fl{(n{+}1)c}-\fl{nc}$. The word $(t(n))_{n\ge 1}$ is
\href{https://oeis.org/A080764}{\texttt{A080764}}.
The correspondence between the
eventually periodic continued fraction
$c=[0;1,2,2,2,\ldots]$ and morphism-invariant characteristic
words is classical, see Lothaire~\cite[Chapter~2]{Lothaire}.
The eigenvalue $1+\sqrt{2}$ is the silver ratio that also
governs the defect substitution of
Section~\ref{sec:defect}.

\subsection{Bifurcations}

The two solutions~$a_{\mathrm{G}}$
and~$a_{\mathrm{B}}$ first diverge at
$n=12$ (Section~\ref{sec:intro}), and their
difference
words diverge one step earlier, with
$w_{\mathrm{G}}(11)=0$ and $w_{\mathrm{B}}(11)=1$.
Figure~\ref{fig:oscillation} plots the normalized difference
$(a_{\mathrm{G}}(n)-a_{\mathrm{B}}(n))/n$, illustrating the
distinct asymptotic behavior of the two branches.
This is the smallest instance of a broader
pattern.  For $r=2$, the greedy solution~$a_{\mathrm{G}}$
is one of continuum many monotone solutions of~\eqref{eq:AG},
organized by the companion paper~\cite{Cloitre_tree} into a
binary branching structure.

More precisely, \cite{Cloitre_tree} constructs a family of
monotone solutions of~\eqref{eq:AG}
with $r=2$ in bijection with the Cantor
space~$\{0,1\}^{\mathbb{N}}$, the greedy solution
corresponding to~$0^{\infty}$ and the lazy solution
to~$1^{\infty}$, and characterizes the $2$-regular solutions
inside the tree as those with eventually constant choice
sequences. The Beatty solution~$a_{\mathrm{B}}$ studied here
does not belong to this tree family. It is characterized by
its Sturmian difference word rather than by a branching
pattern. Whether it is the only Sturmian
solution outside the tree is asked in~\cite{Cloitre_tree}.

\begin{figure}[!htbp]
\centering
\includegraphics[width=\textwidth]{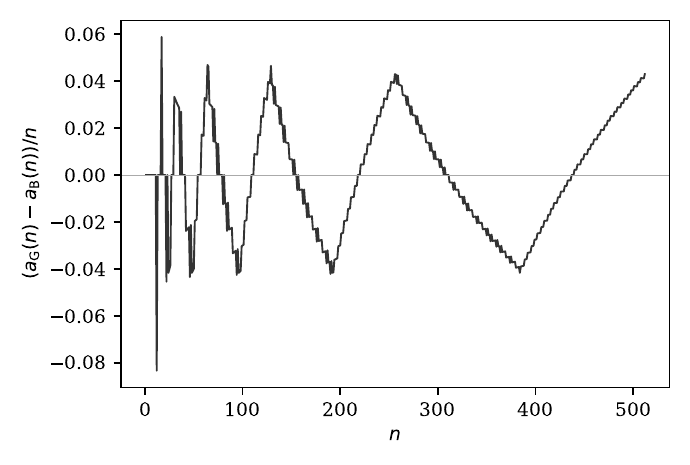}
\caption{The normalized difference
$(a_{\mathrm{G}}(n)-a_{\mathrm{B}}(n))/n$ for $1\le n\le 512$,
illustrating the distinct asymptotic behavior of the greedy and
Beatty branches.}
\label{fig:oscillation}
\end{figure}

\section{The defect at the right endpoint}
\label{sec:defect}

Throughout this section fix the right endpoint shift
$d_{\mathrm{R}}=2(\sqrt{2}-1)$ and write
$a_{\mathrm{R}}(n)=\fl{cn+d_{\mathrm{R}}}$ for the corresponding Beatty
sequence, whose window sum is
$S(n)=a_{\mathrm{R}}(n)+a_{\mathrm{R}}(n{-}1)$. It is distinct from the
canonical solution $a_{\mathrm{B}}(n)=\fl{(n{+}1)/\!\sqrt{2}}$
of Section~\ref{sec:structure}, which has $d=\sqrt{2}/2$ and
satisfies the strong equation with no defect.

At the right endpoint $d=2(\sqrt{2}{-}1)$ of the
interval~$I$, the triple-nested equation holds but
the strong equation $a_{\mathrm{R}}(S(n))=n$ fails
at a sparse set of positions which we call the
\emph{defect set}~$\mathcal{D}$. This section
describes the structure of~$\mathcal{D}$ in three
steps. First we characterize~$\mathcal{D}$ as the
set of return times of the Sturmian rotation
$\theta\mapsto\frc{\theta+c}$ to the interval
$[1-c,1/2]$ (Theorem~\ref{thm:absorption}),
which explains why every defect position is a
repeat of~$a_{\mathrm{R}}$, so that the extra
increment in $a_{\mathrm{R}}(S(n))$ is absorbed
by the outer~$a$.
Second, a direct first-return analysis, together with the
return-time form of the three-distance theorem, gives the
three gap values between consecutive defect
positions, namely~$3$, $4$, and~$7$. Third, the
gap sequence on the alphabet $\{3,4,7\}$ generates the
minimal subshift of an explicit primitive substitution with
Perron eigenvalue~$1+\sqrt{2}$
(Theorem~\ref{thm:defect_subst}), tying the defect structure
to the Pell recurrence.

For $d=2(\sqrt{2}{-}1)$, the triple-nested equation
$a_{\mathrm{R}}(a_{\mathrm{R}}(S(n)))=a_{\mathrm{R}}(n)$
holds by Theorem~\ref{thm:interval}, and expansion
of the defining formula shows
$a_{\mathrm{R}}(S(n))\in\{n,\,n{+}1\}$. The signed defect
of Section~\ref{ssec:composed},
\[
  \mu(n)=a_{\mathrm{R}}(S(n))-n,
\]
therefore takes only the values $0$ and $1$ here, and we write
$\mathcal{D}=\{n\ge 1:\mu(n)=1\}$ for the defect set.
The first elements of~$\mathcal{D}$
(\href{https://oeis.org/A395251}{\texttt{A395251}}) are
\[
  6,\,9,\,13,\,16,\,23,\,30,\,33,\,40,\,47,\,50,\,
  54,\,57,\,64,\,67,\,71,\,\ldots
\]

Write $\theta_{n}=\frc{cn+d}\in[0,1)$, so that
$a_{\mathrm{R}}(n)=cn+d-\theta_{n}$.

\subsection{Defect positions as return times}

The next theorem identifies the defect positions as the
return times of the Sturmian rotation to an explicit
interval, and shows that every defect position is a repeat
of~$a_{\mathrm{R}}$.

\begin{theorem}\label{thm:absorption}
For $d=2(\sqrt{2}{-}1)$ and $n\ge 2$:
\[
  \mu(n)=1\iff
  \theta_{n-1}\in[1-c,\,1/2].
\]
Consequently:
\begin{enumerate}
\item[\textup{(a)}] Every $n\in\mathcal{D}$ is a repeat
  of~$a_{\mathrm{R}}$, i.e.
  $a_{\mathrm{R}}(n{+}1)=a_{\mathrm{R}}(n)$.
\item[\textup{(b)}] The set~$\mathcal{D}$ has natural
  density
\[
  \rho(\mathcal{D})=c-\tfrac{1}{2}
  =\frac{\sqrt{2}-1}{2}\approx 0.2071.
\]
\end{enumerate}
\end{theorem}

\begin{proof}
Using $a_{\mathrm{R}}(n)=cn+d-\theta_{n}$ and
$a_{\mathrm{R}}(n{-}1)=c(n{-}1)+d-\theta_{n-1}$:
\[
  S(n)=a_{\mathrm{R}}(n)+a_{\mathrm{R}}(n{-}1)
  =c(2n{-}1)+2d-\theta_{n}-\theta_{n-1}.
\]
Since $c^{2}=1/2$:
\[
  c\,S(n)+d=n-\tfrac{1}{2}+d(2c{+}1)
  -c(\theta_{n}+\theta_{n-1}).
\]
At $d=2(\sqrt{2}{-}1)=4c-2$,
\[
  d(2c{+}1)=(4c-2)(2c{+}1)
  =8c^{2}-2=2,
\]
so
\begin{equation}\label{eq:cSn_formula}
  c\,S(n)+d=n+\tfrac{3}{2}-c(\theta_{n}+\theta_{n-1}).
\end{equation}
Applying the floor, $a_{\mathrm{R}}(S(n))=n+1$
precisely when $n{+}1\le c\,S(n)+d<n+2$, which using
\eqref{eq:cSn_formula} rearranges to
\[
  -\tfrac{1}{2}\le -c(\theta_{n}+\theta_{n-1})
  <\tfrac{1}{2},
\]
equivalently $\theta_{n}+\theta_{n-1}\le 1/(2c)=c$
(the lower bound is automatic).

Now $\theta_{n}=\frc{\theta_{n-1}+c}$. Since $c<1$, two
cases arise.

If $\theta_{n-1}<1-c$,
then $\theta_{n}=\theta_{n-1}+c$, and
\[
  \theta_{n}+\theta_{n-1}=2\theta_{n-1}+c\le c
\]
forces $\theta_{n-1}=0$. This cannot occur. With $d=4c-2$,
$\theta_{n-1}=0$ would give $c(n-1)+d=c(n+3)-2\in\mathbb{Z}$,
impossible for $n\ge 2$ since $n+3\ne 0$ and $c\notin\mathbb{Q}$.

If $\theta_{n-1}\ge 1-c$,
then $\theta_{n}=\theta_{n-1}+c-1$, and
\[
  \theta_{n}+\theta_{n-1}=2\theta_{n-1}+c-1\le c
\]
becomes $\theta_{n-1}\le 1/2$, yielding
$\theta_{n-1}\in[1-c,1/2]$.

This interval is non-empty since $1/2>1-c$, with
length $1/2-(1-c)=c-1/2=(\sqrt{2}{-}1)/2$. The
characterization is established.

For \textup{(a)}, in the effective case
$\theta_{n}=\theta_{n-1}+c-1\in[0,c-1/2]$.
Since $c-1/2<1-c$, we have $\theta_{n}<1-c$, so
$\theta_{n}+c<1$ and
$a_{\mathrm{R}}(n{+}1)-a_{\mathrm{R}}(n)
=\fl{\theta_{n}+c}=0$.

For \textup{(b)}, the sequence
$(\theta_{n})_{n\ge 0}$ is the orbit of the
irrational rotation $\theta\mapsto\frc{\theta+c}$,
uniformly equidistributed in $[0,1)$ by Weyl's
theorem. The set of $n$ with
$\theta_{n-1}\in[1-c,1/2]$ therefore has natural
density equal to the length of that interval.
\end{proof}

The absorption now reads explicitly. At a
defect position, the point~$\theta_{n}$
lies in the low interval $[0,c-1/2]$, forcing
$a_{\mathrm{R}}(n{+}1)=a_{\mathrm{R}}(n)$. The
equation~\eqref{eq:main} then survives because the
extra~$+1$ in $a_{\mathrm{R}}(S(n))$ is absorbed by
the outer application of~$a_{\mathrm{R}}$,
\[
  a_{\mathrm{R}}(a_{\mathrm{R}}(S(n)))
  =a_{\mathrm{R}}(n{+}1)=a_{\mathrm{R}}(n).
\]

\subsection{The gap sequence}

The characterization
$\mathcal{D}=\{n\ge 2:\theta_{n-1}\in[1-c,1/2]\}$
expresses $\mathcal{D}$ as the return-time set of
the rotation $\theta\mapsto\frc{\theta+c}$ to the
interval $[1-c,1/2]$. By a standard return-time form of the
three-distance theorem~\cite{Sos} (see Alessandri and
Berth\'e~\cite{AlessandriBerthe} for a modern
exposition within combinatorics on words),
the first-return time to the interval takes at most three
values. The
return-time analysis of
Proposition~\ref{prop:defect_substitution} below
identifies these values as $3$, $4$, and~$7$, and
computes the corresponding letter frequencies.

\begin{proposition}\label{prop:defect_substitution}
The gap sequence of~$\mathcal{D}$ on the alphabet
$\{3,4,7\}$ has letter frequencies
\[
  f_{3}=\sqrt{2}-1,\quad f_{4}=3-2\sqrt{2},\quad f_{7}=\sqrt{2}-1,
\]
matching Theorem~\ref{thm:absorption}(b), and the
average gap is $2(\sqrt{2}+1)$. The gap sequence
of~$\mathcal{D}$ is
\href{https://oeis.org/A395253}{\texttt{A395253}}.
\end{proposition}

\begin{proof}
The defect set is
$\mathcal{D}=\{n\ge 2:\theta_{n-1}\in J\}$ with
$J=[1-c,\,1/2]$, so $|J|=c-1/2=(\sqrt{2}-1)/2$.
For each return time $i\ge 1$, let
\[
  K_{i}=\{\theta\in J: R_{c}^{i}\theta\in J
  \text{ and } R_{c}^{k}\theta\notin J
  \text{ for }1\le k<i\},
\]
where $R_{c}\colon\theta\mapsto\frc{\theta+c}$.
We compute the sub-intervals directly from the
iterates. For $\theta\in J=[1-c,1/2]$, a direct reduction
modulo~$1$ gives
\[
  R_{c}\theta=\theta+c-1,\qquad
  R_{c}^{2}\theta=\theta+2c-1,\qquad
  R_{c}^{3}\theta=\theta+3c-2
\]
for $\theta\in J$. Neither $R_{c}\theta\in[0,c-1/2]$
nor $R_{c}^{2}\theta\in[c,2c-1/2]$ meets~$J$, so
every return time is at least~$3$.

Return time~$3$ occurs when
$\theta+3c-2\in[1-c,1/2]$, equivalently
$\theta\in[3-4c,\,5/2-3c]$. Since $1-c>3-4c$
and $5/2-3c<1/2$, intersection with~$J$ gives
\[
  K_{3}=[1-c,\,5/2-3c],\qquad
  |K_{3}|=(5/2-3c)-(1-c)=3/2-2c=\tfrac{3-2\sqrt{2}}{2}.
\]

For $\theta\in J\setminus K_{3}=(5/2-3c,\,1/2]$, the
return time exceeds~$3$. The next iterate
$R_{c}^{4}\theta=\theta+4c-3$ lies in~$J$ when
$\theta\in[4-5c,\,7/2-4c]$. Intersection with
$(5/2-3c,\,1/2]$ gives
\[
  K_{4}=[4-5c,\,1/2],\qquad
  |K_{4}|=1/2-(4-5c)=5c-7/2=\tfrac{5\sqrt{2}-7}{2}.
\]

The remainder
$K_{7}=(5/2-3c,\,4-5c)=J\setminus(K_{3}\sqcup K_{4})$
has length
\[
  |K_{7}|=(4-5c)-(5/2-3c)=3/2-2c=\tfrac{3-2\sqrt{2}}{2},
\]
For $\theta\in K_{7}$, the fifth and sixth iterates satisfy
\[
  R_{c}^{5}\theta=\theta+5c-3\in\bigl(2c-\tfrac12,\,1\bigr),
  \qquad
  R_{c}^{6}\theta=\theta+6c-4\in\bigl(3c-\tfrac32,\,c\bigr),
\]
both disjoint from $J=[1-c,1/2]$, since $3c-\tfrac32>\tfrac12$.
On the other hand,
\[
  R_{c}^{7}\theta=\theta+7c-5\in\bigl(4c-\tfrac52,\,2c-1\bigr)
  \subset J,
\]
so the first-return time is $7$ throughout $K_{7}$, and
$J=K_{3}\sqcup K_{7}\sqcup K_{4}$ is the exact first-return
partition.

The orbit $(\theta_{n-1})_{n\ge 2}$ equidistributes
in $[0,1)$ by Weyl's theorem, so each $K_{i}$ is
visited with density~$|K_{i}|$, and the frequency
of gap value~$i$ among defect positions is
\[
  f_{i}=\frac{|K_{i}|}{|J|}.
\]
Substituting the lengths and using
$(\sqrt{2}-1)(\sqrt{2}+1)=1$ to rationalize yields
\[
  f_{3}=f_{7}=\sqrt{2}-1,\qquad
  f_{4}=3-2\sqrt{2}.
\]
The average gap is
\[
  3f_{3}+4f_{4}+7f_{7}
  =3(\sqrt{2}{-}1)+4(3{-}2\sqrt{2})+7(\sqrt{2}{-}1)
  =2\sqrt{2}+2
  =2(\sqrt{2}{+}1),
\]
which equals the reciprocal of the density
$\rho(\mathcal{D})=(\sqrt{2}-1)/2$ established in
Theorem~\ref{thm:absorption}(b).
\end{proof}

We now come to the substitution structure of the gap
sequence. Let $T\colon J\to J$ denote the first-return map
of $R_{c}$ to $J$. On the partition of
Proposition~\ref{prop:defect_substitution}, $T$ acts by
\[
  T\theta=\theta+3c-2 \text{ on } K_{3},\qquad
  T\theta=\theta+4c-3 \text{ on } K_{4},\qquad
  T\theta=\theta+7c-5 \text{ on } K_{7}.
\]
For $\theta\in J$, the itinerary
$C(\theta)=(\iota(T^{k}\theta))_{k\ge 0}$, where
$\iota(\theta')=i$ when $\theta'\in K_{i}$, records the
successive gap values along the orbit. The gap sequence
of~$\mathcal{D}$ is the itinerary of the first orbit point
in~$J$.

\begin{lemma}\label{lem:selfind}
Let $\eta=\sqrt{2}-1=\tau^{-1}$, and define the
orientation-reversing similarity
\[
  \psi\colon J\to K_{3},\qquad
  \psi(\theta)=\tfrac{5-3\sqrt{2}}{2}
  -\eta\bigl(\theta-(1-c)\bigr).
\]
Set $X_{1}=\tfrac{19-13\sqrt{2}}{2}$ and
$X_{2}=6-4\sqrt{2}$. Then the following hold.
\begin{enumerate}
\item[\textup{(a)}] $\psi(K_{4})=[1{-}c,\,X_{1}]$,
$\psi(K_{7})=(X_{1},\,X_{2})$, and
$\psi(K_{3})=[X_{2},\,\tfrac{5-3\sqrt{2}}{2}]$.
\item[\textup{(b)}] The points of
$[X_{2},\,\tfrac{5-3\sqrt{2}}{2}]$ return to $K_{3}$ after
the itinerary $3,4$, the points of $(X_{1},\,X_{2})$ after
the itinerary $3,7,7$, and the points of $[1{-}c,\,X_{1}]$
after the itinerary $3,7$.
\item[\textup{(c)}] For every $\theta\in K_{i}$,
\[
  \psi(T\theta)=T^{\rho(i)}\bigl(\psi(\theta)\bigr),
  \qquad \rho(3)=2,\quad\rho(4)=2,\quad\rho(7)=3.
\]
\end{enumerate}
\end{lemma}

\begin{proof}
(a) The map $\psi$ is affine, so it suffices to evaluate it
at the endpoints $1-c$, $\tfrac{5-3\sqrt{2}}{2}$, $4-5c$,
and $\tfrac{1}{2}$ of the partition, using $c=\sqrt{2}/2$.
The computations
$\psi(\tfrac12)=1-c$, $\psi(4-5c)=X_{1}$,
$\psi(\tfrac{5-3\sqrt{2}}{2})=X_{2}$, and
$\psi(1-c)=\tfrac{5-3\sqrt{2}}{2}$
are direct, and $\psi$ reverses orientation.

(b) We follow the orbit of each piece. First,
$T\bigl([X_{2},\tfrac{5-3\sqrt{2}}{2}]\bigr)
=[4-5c,\tfrac12]=K_{4}$, since
$X_{2}+3c-2=4-5c$ and $\tfrac{5-3\sqrt{2}}{2}+3c-2=\tfrac12$,
and $T(K_{4})=[1-c,\,4c-\tfrac52]\subset K_{3}$ as in the
proof of Proposition~\ref{prop:defect_substitution}. This
gives the itinerary $3,4$ with return after two steps.
Second,
$T\bigl([1{-}c,X_{1}]\bigr)=[2c-1,\,\tfrac{15}{2}-5\sqrt{2}]
\subset K_{7}$ and one further step lands in
$[9c-6,\,\tfrac{5-3\sqrt{2}}{2}]\subset K_{3}$, giving the
itinerary $3,7$. Third,
$T\bigl((X_{1},X_{2})\bigr)
=(\tfrac{15}{2}-5\sqrt{2},\,4-5c)\subset K_{7}$, the next
step gives $(\tfrac{5-3\sqrt{2}}{2},\,2c-1)\subset K_{7}$,
and the third step gives
$(4c-\tfrac52,\,9c-6)\subset K_{3}$, giving the itinerary
$3,7,7$. All endpoint computations take place in
$\mathbb{Q}(\sqrt{2})$.

(c) On each piece, both sides of the identity are
translations of $\psi(\theta)$, so it suffices to compare
the translation amounts. For $\theta\in K_{3}$, the left
side translates $\psi(\theta)$ by $-\eta(3c-2)$ and the
right side by $(3c-2)+(4c-3)=7c-5$, and
$(\sqrt{2}-1)(3c-2)=5-7c$. For $\theta\in K_{7}$, the
amounts are $-\eta(7c-5)$ and $(3c-2)+2(7c-5)=17c-12$,
and $(\sqrt{2}-1)(7c-5)=12-17c$. For $\theta\in K_{4}$, the
amounts are $-\eta(4c-3)$ and $(3c-2)+(7c-5)=10c-7$, and
$(\sqrt{2}-1)(4c-3)=7-10c$. Each of the three products is
verified by expanding with $c=\sqrt{2}/2$.
\end{proof}

The self-induction of the return map is realized by an
explicit substitution on the three gap values.

\begin{theorem}\label{thm:defect_subst}
Define the substitution
\begin{equation}\label{eq:sigma}
  \sigma\colon\quad
  3\mapsto 3\,4,\qquad
  4\mapsto 3\,7,\qquad
  7\mapsto 3\,7\,7,
\end{equation}
with incidence matrix
\[
  M=\begin{pmatrix}
  1 & 1 & 1\\
  1 & 0 & 0\\
  0 & 1 & 2
  \end{pmatrix},
\]
characteristic polynomial
$\lambda^{3}-3\lambda^{2}+\lambda+1
=(\lambda-1)(\lambda^{2}-2\lambda-1)$, and Perron eigenvalue
$1+\sqrt{2}$. Every entry of $M^{2}$ is positive, so
$\sigma$ is primitive. The gap sequence of~$\mathcal{D}$
belongs to the minimal subshift $X_{\sigma}$ generated
by~$\sigma$, and its letter frequencies
$(\sqrt{2}-1,\,3-2\sqrt{2},\,\sqrt{2}-1)$ equal the
normalized Perron eigenvector of~$M$.
\end{theorem}

\begin{proof}
Lemma~\ref{lem:selfind} gives, for every $\theta\in J$,
\begin{equation}\label{eq:conj_coding}
  C(\psi\theta)=\sigma\bigl(C(\theta)\bigr).
\end{equation}
Indeed, parts (a) and (b) show that $\psi\theta$ begins the
itinerary $\sigma(\iota(\theta))$, and part (c) shows that
after this itinerary the orbit of $\psi\theta$ stands at
$\psi(T\theta)$, so induction on the returns
proves~\eqref{eq:conj_coding}.

Since $X_{\sigma}$ is the set of infinite words all of whose
factors are factors of the fixed point
$\sigma^{\infty}(3)$, it suffices to prove that every factor
of every itinerary is a factor of $\sigma^{\infty}(3)$.

The map $T$ permutes the pieces as follows. From the proof
of Proposition~\ref{prop:defect_substitution},
$T(K_{3})=[2c-1,\tfrac12]$ meets $K_{7}$ and $K_{4}$ but not
$K_{3}$, $T(K_{4})\subset K_{3}$, and
$T(K_{7})=(4c-\tfrac52,\,2c-1)$ meets $K_{3}$ and $K_{7}$
but not $K_{4}$. Moreover
$T\bigl(T(K_{4})\bigr)\subset K_{7}$, since
$T(K_{4})=[1-c,\,4c-\tfrac52]$ translates by $3c-2$ into
$[2c-1,\,7c-\tfrac92]\subset K_{7}$, and
$T\bigl(T(K_{7})\cap K_{7}\bigr)\subset K_{3}$, since
$(\tfrac{5-3\sqrt{2}}{2},\,2c-1)$ translates by $7c-5$ into
$(4c-\tfrac52,\,9c-6)\subset K_{3}$. Consequently the
factors of length two of any itinerary lie in
$\{34,37,43,73,77\}$, the word $43$ is always followed
by~$7$, and the word $77$ is always followed by~$3$. The
factors of length three therefore lie in
$\{343,373,377,437,734,737,773\}$, and each of these words
is a factor of $\sigma^{4}(3)$, by inspection.

Now let $F$ be a factor of an itinerary $C(\theta)$, of
length $\ell\ge 4$. Shifting $\theta$ along its orbit, $F$
is a factor of $C(\theta_{0})$ for a point $\theta_{0}$
whose backward orbit meets $K_{3}$ within two steps, because
the return itineraries of Lemma~\ref{lem:selfind}(b) spend
at most two consecutive steps outside $K_{3}$. Hence $F$ is
a factor of $C(\theta_{3})$ for some $\theta_{3}\in K_{3}$,
and $C(\theta_{3})=\sigma(C(\theta_{4}))$ with
$\theta_{4}=\psi^{-1}(\theta_{3})$
by~\eqref{eq:conj_coding}. Every factor of length $\ell$ of
$\sigma(W)$ is contained in $\sigma(F')$ for a factor $F'$
of $W$ of length at most $\lceil\ell/2\rceil+1<\ell$, since
each letter image has length at least two. By induction on
$\ell$, the factor $F'$ is a factor of
$\sigma^{\infty}(3)$, hence so is $F$.

The letter frequencies are computed in
Proposition~\ref{prop:defect_substitution}, and the
normalized Perron eigenvector of $M$ is
$(\sqrt{2}-1,\,3-2\sqrt{2},\,\sqrt{2}-1)$ by a direct
computation.
\end{proof}

\begin{corollary}\label{cor:343}
In the gap sequence of~$\mathcal{D}$, the factors of length
two are exactly $34$, $37$, $43$, $73$, and $77$. In
particular, every occurrence of the letter~$4$ is preceded
and followed by the letter~$3$.
\end{corollary}

\begin{proof}
The inclusion is the pair analysis in the proof of
Theorem~\ref{thm:defect_subst}. Each of the five words
occurs in $\sigma^{3}(3)=343734377$,
for~$\sigma$ as in~\eqref{eq:sigma}, and $X_{\sigma}$ is
minimal because $\sigma$ is primitive, so every element of
$X_{\sigma}$ has the same factor set, and the five words
occur in the gap sequence.
\end{proof}

\begin{remark}
Identity~\eqref{eq:conj_coding} is an exact self-induction.
The first-return map of $T$ on $K_{3}$ is conjugate to $T$
by the contraction $\psi$ of ratio
$\sqrt{2}-1=1/(1+\sqrt{2})$, so the expansion factor is the
Perron eigenvalue $1+\sqrt{2}$, the
silver ratio of the Pell recurrence of
Section~\ref{sec:pell_ostrowski}. The transition
computations agree with the gap sequence of $\mathcal{D}$
computed to $2\cdot 10^{6}$ terms.
\end{remark}

\begin{remark}
The three gap values $3,4,7$ are exactly
$q_{2}$, $q_{3}-q_{2}$, and $q_{3}$ in the sequence
$(q_{k})$ of convergent denominators, and the Perron eigenvalue
$1+\sqrt{2}$ of~$\sigma$ is the dominant root of
the Pell recurrence $X_{n+1}=2X_{n}+X_{n-1}$. Both
facts
instantiate the Pell--Ostrowski structure described
in Section~\ref{sec:pell_ostrowski}.
\end{remark}

\subsection{A two-letter coarsening}

Encoding the gap values $3$ and $4$ as a letter
$\mathtt{S}$ and the gap value $7$ as a letter $\mathtt{L}$
yields a two-letter sequence with frequencies $2-\sqrt{2}$
for~$\mathtt{S}$ and $\sqrt{2}-1$ for~$\mathtt{L}$. This
sequence is not Sturmian. The words $\mathtt{SS}$ and
$\mathtt{LL}$ both occur, by Corollary~\ref{cor:343} applied
to the factors $34$ and $77$, so two factors of length two
differ by two occurrences of~$\mathtt{L}$, and the sequence
is not balanced.
A direct computation from Theorem~\ref{thm:defect_subst}
gives factor complexity
\[
  p(\ell)=2\ell+1\qquad(3\le\ell\le 15).
\]
These initial values are compatible with the linear
complexity expected for the coding of the return map $T$ by
the disconnected atom $K_{3}\cup K_{4}$, whose parts are not
unions of adjacent intervals. We do not claim that the formula
$p(\ell)=2\ell+1$ holds for every~$\ell$.

\section{Open problems}\label{sec:open}
\begin{enumerate}
\item Is the affine solution~\eqref{eq:fsol}
  the unique continuous
  monotone solution of the continuous
  equation~\eqref{eq:continuous}?
  By Theorem~\ref{thm:continuous}, any such
  solution is a homeomorphism satisfying
  $\varphi^{-1}(x)=\sum_{j=0}^{r-1}\varphi(x{-}j)$.
\item The admissible shifts $W_{r}^{(k)}$ of the iterated
  equation $(E_{k})$ form an increasing chain
  $W_{r}^{(1)}\subseteq W_{r}^{(2)}\subseteq\cdots$
  (Corollary~\ref{cor:chain}). For non-square~$r$ one has
  $W_{r}^{(1)}=J_{r}$, while $W_{2}^{(2)}=I$ and
  $W_{3}^{(2)}=I_{3}$. Is each
  $W_{r}^{(k)}$ an interval? Does the chain become constant
  after finitely many steps, and if so, can the eventual set
  $W_{r}^{(\infty)}=\bigcup_{k\ge 1}W_{r}^{(k)}$ and the index
  at which it is reached be given in closed form? By
  Proposition~\ref{prop:composed} these questions ask for
  which shifts $a^{\circ(k-1)}$ takes the same value at $n$
  and at $n+\mu(n)$ for every~$n\ge r$.
\item Among nondecreasing solutions
  of~\eqref{eq:AG} for $r\ge 3$, classify all
  monotone solutions.  For $r=2$ the tree
  structure is described in~\cite{Cloitre_tree}.
  The general case is open.
\item In the Ostrowski numeration system
  for $c=[0;1,2,2,2,\ldots]$
  (see~\cite{Shallit_book}),
  with partial quotients $b_{k}$, convergents $p_{k}/q_{k}$,
  and digits $0\le d_{k}\le b_{k+1}$ subject to the usual
  admissibility rule $d_{k}=b_{k+1}\Rightarrow d_{k-1}=0$,
  prove that
  $a_{\mathrm{B}}(n)=\sum d_{k}p_{k}$
  whenever $n=\sum d_{k}q_{k}$.
  This digit-swap property has been
  confirmed numerically for $n\le 10^{6}$
  and would give a carry-free interpretation
  of the equation
  $a_{\mathrm{B}}(a_{\mathrm{B}}(n)+a_{\mathrm{B}}(n{-}1))=n$.
\item Identify the analogue of
  Theorem~\ref{thm:defect_subst} for each
  $r\ge 3$. For $r=3$ the defect set at the
  right endpoint $d=(5\sqrt{3}{-}3)/6$ begins
\[
  16,\,23,\,42,\,49,\,61,\,68,\,87,\,94,\ldots
\]
  (\href{https://oeis.org/A395252}{\texttt{A395252}}), with
  numerically observed gap
  values $\{7,\,12,\,19\}$ corresponding to the
  consecutive convergent denominators $7$ and
  $19$ of~$1/\!\sqrt{3}$ and their difference.
  A primitive substitution on~$\{7,12,19\}$
  should govern the structure, with Perron
  eigenvalue related to~$\sqrt{3}$.
\end{enumerate}


\appendix

\section{Analytic oscillation estimates}
\label{app:analytic}

This appendix carries out the block estimates behind
Lemmas~\ref{lem:boundary} and~\ref{lem:middle}. Throughout,
$r=s^{2}+t$ is non-square with $s=\fl{\sqrt{r}}$,
$\alpha=\sqrt{r}$, $\beta=\alpha-s$, and $u=\fl{t/2}$, and
$E$, $H$, $P$ are as in Section~\ref{sec:universal}. By
Lemma~\ref{lem:blocks} and Lemma~\ref{lem:window}(a), the
bound $\max P-\min P<\alpha-1$ follows from
inequality~\eqref{eq:boundary_target}.

\begin{proof}[Proof of Lemma~\ref{lem:boundary}]
We evaluate $\operatorname{osc}E$ exactly in the four cases.

Suppose first that $t\le 2$. Then
$\beta=t/(s+\alpha)<1/s$, so $\frc{j\beta}=j\beta$ for
$0\le j\le s-1$ and $\frc{j\beta}=1+j\beta$ for
$-s\le j\le-1$. Summing over the window,
\[
  E(A)=(s-A)+\beta\,\frac{s(2A-s-1)}{2},
  \qquad 0\le A\le s,
\]
a linear function of $A$ with slope $s\beta-1<0$, so
$\operatorname{osc}E=s(1-s\beta)$. For $t=1$, where $u=0$,
inequality~\eqref{eq:boundary_target} reduces to
$\beta(s^{2}+1)>2$, that is, $s^{2}+1>2(s+\alpha)$. This
holds for $s\ge 5$ because $\alpha<s+1$, and for $s=4$
because $17>8+2\sqrt{17}$. The case $s=3$ is the excluded
value $r=10$. For $t=2$, where $u=1$,
inequality~\eqref{eq:boundary_target} reduces to
$\beta(s^{2}+2)>3$, that is, $2(s^{2}+2)>3(s+\alpha)$,
which holds for $s\ge 3$ because $\alpha<s+1$ and
$2s^{2}-6s+1>0$.

Suppose now that $t\ge 2s-1$, and set $\gamma=1-\beta$.
Then $s\gamma<1$. Indeed, for $t=2s$,
\[
  \gamma=(s+1)-\sqrt{s^{2}+2s}
  =\frac{1}{(s+1)+\sqrt{s^{2}+2s}}<\frac{1}{2s+1},
\]
and for $t=2s-1$,
$\gamma=2/\bigl((s+1)+\sqrt{s^{2}+2s-1}\bigr)<1/s$, using
$(s-1)^{2}<s^{2}+2s-1$. Since
$\frc{j\beta}=1-\frc{j\gamma}$ for $j\neq 0$, and the
window of $E(A)$ contains $j=0$ exactly when
$1\le A\le s$, the computation of the first case applied
to~$\gamma$ gives
\[
  \begin{aligned}
    E(A)&=(A-1)-\gamma\,\frac{s(2A-s-1)}{2}
    \quad(1\le A\le s),\\
    E(0)&=\sum_{j=1}^{s}j\gamma=\gamma\,\frac{s(s+1)}{2}.
  \end{aligned}
\]
On $[1,s]$ the function $E$ is linear with positive slope
$1-s\gamma$, and $E(0)-E(1)=\gamma s>0$, so the minimum is
$E(1)$ and
\[
  \operatorname{osc}E
  =\max\bigl(E(0),E(s)\bigr)-E(1)
  =\max\bigl(\gamma s,\;(s-1)(1-s\gamma)\bigr).
\]
For $t=2s-1$, where $u+1=s$,
inequality~\eqref{eq:boundary_target} splits into the two
branches $2\gamma s<s-1$, which holds because $\gamma s<1$
and $s\ge 3$, and
$(s-1)(1-s\gamma)+\gamma s<s-1$, which is equivalent to
$\gamma s(s-2)>0$. For $t=2s$, where $u+1=s+1$, the two
branches are $\gamma s+\gamma(s+1)<s-1$, which holds
because $\gamma(2s+1)<1\le s-1$, and
$(s-1)(1-s\gamma)+\gamma(s+1)<s-1$, which is equivalent to
$s+1<s(s-1)$, valid for $s\ge 3$.

In every case inequality~\eqref{eq:boundary_target} holds,
and Proposition~\ref{prop:oscillation} establishes the equation.
\end{proof}

The middle range uses the following block-deviation estimate.

\begin{lemma}\label{lem:blockdev}
Let $\gamma\in(0,1/2]$ be irrational, and partition the
integers into the blocks $\{j:\fl{j\gamma}=m\}$,
$m\in\mathbb{Z}$. Set
\[
  K(\gamma)=
  \begin{cases}
    1-\gamma, & 0<\gamma\le 1/3,\\[2pt]
    (1+\gamma)^{2}/(8\gamma), & 1/3\le\gamma\le 1/2.
  \end{cases}
\]
\begin{enumerate}
\item[\textup{(a)}] Each block is a set of consecutive
integers on which the values $\frc{j\gamma}$ form an
arithmetic progression with step $\gamma$, first value
$\delta\in[0,\gamma)$, and length $L$ satisfying
$\gamma L=1-\eta$ with $|\eta|<\gamma$ and
$\delta<\gamma+\eta$.
\item[\textup{(b)}] For every block,
$\bigl|\sum_{j\in\mathrm{block}}
(\frc{j\gamma}-\tfrac12)\bigr|\le K(\gamma)$.
\item[\textup{(c)}] For every set of $\ell$ consecutive
integers contained in one block,
$\bigl|\sum(\frc{j\gamma}-\tfrac12)\bigr|
\le(1+\gamma)^{2}/(8\gamma)$.
\end{enumerate}
\end{lemma}

\begin{proof}
(a) The block with $\fl{j\gamma}=m$ is the integer
interval $[\lceil m/\gamma\rceil,\,\allowbreak
\lceil(m{+}1)/\gamma\rceil)$, on which
$\frc{j\gamma}=j\gamma-m$ increases with step $\gamma$
from $\delta=\lceil m/\gamma\rceil\gamma-m\in[0,\gamma)$.
The length satisfies $1/\gamma-1<L<1/\gamma+1$, so
$\eta=1-\gamma L\in(-\gamma,\gamma)$. The last value of
the block is $\delta+(L-1)\gamma<1$, which reads
$\delta<\gamma+\eta$.

(b) The block sum is
$L\delta+\gamma L(L-1)/2-L/2
=L\bigl(\delta-(\gamma+\eta)/2\bigr)$,
and $L=(1-\eta)/\gamma$. For $\eta\ge 0$ the value
$\delta$ ranges over $[0,\gamma)$, so the block sum is at
most $h(\eta):=(1-\eta)(\gamma+\eta)/(2\gamma)$
in absolute value. On $[0,\gamma)$ the function $h$
increases up to $\eta^{*}=(1-\gamma)/2$ and decreases
afterwards. If $\gamma\le 1/3$ then $\eta^{*}\ge\gamma$
and $h<h(\gamma)=1-\gamma$. If
$\gamma>1/3$ then
$h\le h(\eta^{*})=(1+\gamma)^{2}/(8\gamma)$.
For $\eta<0$ the value $\delta$ ranges over
$[0,\gamma+\eta)$, so the block sum is again at most
$h(\eta)$ in absolute value, and $h$ is
increasing on $(-\gamma,0)$, so
$h<h(0)=1/2\le K(\gamma)$.

(c) The $\ell$ values form an arithmetic progression with
step $\gamma$ inside $[0,1)$, so their mean $\bar{x}$
satisfies
$\gamma(\ell-1)/2\le\bar{x}<1-\gamma(\ell-1)/2$, and
the sum deviates from $\ell/2$ by at most
$(\ell/2)\bigl(1-\gamma(\ell-1)\bigr)$. As a function
of the real variable $\ell$, this expression is maximal at
$\ell=(1+\gamma)/(2\gamma)$, with value
$(1+\gamma)^{2}/(8\gamma)$.
\end{proof}

\begin{proof}[Proof of Lemma~\ref{lem:middle}]
Set $\gamma=\min(\beta,1-\beta)$. Both quantities are
bounded below on the middle range. Indeed
$\beta=t/(s+\alpha)\ge 3/(s+\alpha)>3/(2s+2)$, and
$1-\beta=(2s+1-t)/\bigl((s+1)+\alpha\bigr)
\ge 3/\bigl((s+1)+\alpha\bigr)>3/(2s+2)$, so
\begin{equation}\label{eq:gamma_lb}
  \frac{3}{2s+2}<\gamma\le\frac{1}{2}.
\end{equation}
The target inequality is~\eqref{eq:boundary_target}.

First, $\operatorname{osc}E$ is bounded through the blocks
of the $\gamma$-rotation. If $\beta>1/2$, then
$\frc{j\beta}=1-\frc{j\gamma}$ for $j\neq 0$, so
$E(A)=s-E_{\gamma}(A)-\chi(A)$, where $E_{\gamma}$
denotes the window sum of the $\gamma$-rotation and
$\chi(A)\in\{0,1\}$ records whether the window
contains $j=0$. Hence
$\operatorname{osc}E\le\operatorname{osc}E_{\gamma}+1$,
and this inequality also holds trivially when
$\beta\le 1/2$. Each window of length $s$ meets at most
$s/\fl{1/\gamma}$ full blocks, which are disjoint
intervals of length at least $\fl{1/\gamma}$, and at most
two partial blocks at its ends. By
Lemma~\ref{lem:blockdev},
\[
  \Bigl|E_{\gamma}(A)-\frac{s}{2}\Bigr|
  \le\frac{s}{\fl{1/\gamma}}\,K(\gamma)
  +2\,\frac{(1+\gamma)^{2}}{8\gamma},
\]
so
\begin{equation}\label{eq:oscE_bound}
  \operatorname{osc}E
  \le\frac{2sK(\gamma)}{\fl{1/\gamma}}
  +\frac{(1+\gamma)^{2}}{2\gamma}+1.
\end{equation}

Second, the term $(1-\beta)(u+1)$ is bounded by
\begin{equation}\label{eq:first_bound}
  (1-\beta)(u+1)
  \le(1-\beta)\Bigl(\frac{t}{2}+1\Bigr)
  =s\beta(1-\beta)+\frac{\beta^{2}(1-\beta)}{2}
  +(1-\beta)
  \le s\gamma(1-\gamma)+\frac{9}{8},
\end{equation}
using $t=\beta(2s+\beta)$,
$\beta(1-\beta)=\gamma(1-\gamma)$, and
$\beta^{2}(1-\beta)\le 4/27$.

Set $\Theta(\gamma)=2K(\gamma)/\fl{1/\gamma}
+\gamma(1-\gamma)$.
Combining \eqref{eq:oscE_bound}
and~\eqref{eq:first_bound}, the target
inequality~\eqref{eq:boundary_target} follows from
\begin{equation}\label{eq:final_middle}
  s\,\Theta(\gamma)+\frac{(1+\gamma)^{2}}{2\gamma}
  +\frac{25}{8}<s.
\end{equation}

Third, $\Theta(\gamma)<8/9$ for every
$\gamma\in(0,1/2]$. On $(1/3,1/2]$, where
$\fl{1/\gamma}=2$,
\[
  \Theta(\gamma)
  =\frac{(1+\gamma)^{2}}{8\gamma}+\gamma(1-\gamma),
  \qquad
  \Theta'(\gamma)
  =\frac{1}{8}-\frac{1}{8\gamma^{2}}+1-2\gamma<0,
\]
so $\Theta<\lim_{\gamma\to 1/3^{+}}\Theta=8/9$ there. On
$(1/4,1/3]$, where $\fl{1/\gamma}=3$,
$\Theta=(1-\gamma)(2/3+\gamma)\le 11/16$. For
$\gamma\le 1/4$, the bound
$\fl{1/\gamma}\ge(1-\gamma)/\gamma$ gives
$\Theta\le 2\gamma+\gamma(1-\gamma)\le 3\gamma\le 3/4$.

Finally, \eqref{eq:final_middle} is verified in two
ranges. For $\gamma\ge 1/8$, the map
$\gamma\mapsto(1+\gamma)^{2}/(2\gamma)$ is decreasing,
so the left-hand side of~\eqref{eq:final_middle} is less
than $\tfrac{8}{9}s+\tfrac{81}{16}+\tfrac{25}{8}$, and
$\tfrac{8}{9}s+\tfrac{131}{16}<s$ holds for
$s\ge 74$. For $\gamma<1/8$, we use $\Theta\le 3\gamma$
and $(1+\gamma)^{2}<81/64$
together with~\eqref{eq:gamma_lb}:
\[
  s\,\Theta+\frac{(1+\gamma)^{2}}{2\gamma}+\frac{25}{8}
  <\frac{3s}{8}+\frac{81}{128}\cdot\frac{2s+2}{3}
  +\frac{25}{8}
  =\frac{3s}{8}+\frac{27(s+1)}{64}+\frac{25}{8}<s,
\]
where the last inequality reduces to $13s>227$, valid for
$s\ge 18$. In both ranges
\eqref{eq:final_middle} holds, and
Proposition~\ref{prop:oscillation} establishes the equation.
\end{proof}

\section{Exact finite verification}
\label{app:finite}

This appendix describes the computation behind
Proposition~\ref{prop:finite}. For each non-square $r$ with
$3\le r\le 5473$, the walk values $P(j)$, $0\le j\le r-1$,
are elements $A+B\sqrt{r}$ of $\mathbb{Z}[\sqrt{r}]$. From
$P(j)=j\beta-C(j)$ with $\beta=\sqrt{r}-s$ one has
$A=-(js+C(j))$ and $B=j$, where
$C(j)=\#\{1\le k\le j:\frc{k/\!\sqrt{r}}<\beta\}$ is computed
without floating point. Writing $m_{k}=\fl{k/\!\sqrt{r}}$, one
has $\frc{k/\!\sqrt{r}}<\beta$ if and only if
$r-k>(s-m_{k})\sqrt{r}$, equivalently
$(r-k)^{2}>r(s-m_{k})^{2}$, with strict inequality because $r$
is not a square. Here $m_{k}$ is the unique integer with
$m_{k}^{2}r\le k^{2}<(m_{k}+1)^{2}r$. Two values
$A+B\sqrt{r}$ and $A'+B'\sqrt{r}$ are compared exactly.
Writing $(P,Q)=(A-A',B-B')$, the sign of $P+Q\sqrt{r}$ is read
from the signs of $P$ and $Q$, and when they differ from the
sign of $P^{2}-Q^{2}r$ interpreted according to the sign
of~$P$. If $P>0>Q$ then $P+Q\sqrt{r}>0$ exactly when
$P^{2}>Q^{2}r$, and if $P<0<Q$ the inequality reverses. No
floating-point arithmetic enters any decision.

Taking the maximum and minimum of the $P(j)$ by these
comparisons yields $\operatorname{osc}P=\max P-\min P$ as an
element of $\mathbb{Z}[\sqrt{r}]$, and
$\operatorname{osc}P<\sqrt{r}-1$ is decided by one further
comparison. Each of the $5399$ non-square orders is settled by
a finite sequence of integer comparisons, with no rounding.
The minimum over this range of the margin
$\sqrt{r}-1-\operatorname{osc}P$ is
\[
  3\sqrt{3}-5\approx 0.1962,
\]
attained at $r=3$ and at no other order in the range, so
every margin is positive. That
$3\sqrt{3}-5$ is the exact minimum is itself decided in
$\mathbb{Z}[\sqrt{r}]$. Writing a margin as
$M_{r}=A_{r}+B_{r}\sqrt{r}$, the comparison
$M_{r}\ge 3\sqrt{3}-5$ reduces, since $M_{r}+5>0$, to the sign
of $(A_{r}+5)^{2}+B_{r}^{2}r-27+2B_{r}(A_{r}+5)\sqrt{r}$.

The program \texttt{verify\_certificate.py} that performs
these comparisons accompanies the article as supplementary
material.

\section*{Funding}
The author received no funds, grants, or other support for this
work.

\section*{Data availability}
The integer-arithmetic program of Appendix~\ref{app:finite},
verifying Proposition~\ref{prop:finite} for $3\le r\le 5473$,
is provided with the submission as supplementary material. No
other data are associated with this article.

\section*{Statement on the use of AI}
During the preparation of this work, the author used Claude
(Opus~4.8, Anthropic) to assist with language editing,
manuscript organization, the drafting and checking of
mathematical arguments, and numerical and symbolic
verification. The author reviewed, edited, and independently
verified all definitions, statements, and proofs, and takes
full responsibility for the content of the article.

\end{document}